\documentclass[a4paper, 12 pt, twoside, reqno]{amsart}

\usepackage{slashed}
\usepackage[english]{babel}
\usepackage[utf8]{inputenc}

\usepackage{enumerate}
\usepackage[all]{xy}
\usepackage{setspace}
\usepackage[euler-digits,euler-hat-accent]{eulervm}

\usepackage{times}
\usepackage{amsmath}
\usepackage{amsfonts}
\usepackage{amssymb}
\usepackage{amsmath}
\usepackage{amsthm}
\usepackage{graphicx}
\usepackage{array}
\usepackage{color}
\usepackage{mathrsfs}
\usepackage{hyperref}
\usepackage{eucal}
\usepackage{esint} 
\usepackage{tikz}
\usepackage{upgreek}
\usepackage{bm}

\allowdisplaybreaks

\theoremstyle{plain}
\newtheorem{theorem}{Theorem}[section]
\newtheorem{corollary}[theorem]{Corollary}
\newtheorem{proposition}[theorem]{Proposition}
\newtheorem{lemma}[theorem]{Lemma}

\theoremstyle{definition}
\newtheorem{definition}[theorem]{Definition}
\newtheorem{assumption}[theorem]{Assumption}

\theoremstyle{remark}
\newtheorem{remark}[theorem]{Remark} 
\newtheorem{example}[theorem]{Example}

\numberwithin{equation}{section}
\numberwithin{figure}{section}
\numberwithin{table}{section}

\newcommand{\Om}{{\bf \Omega}}
\newcommand{\R}{\mathbb{R}}
\newcommand{\N}{\mathbb{N}}

\newcommand{\C}{\mathbb{C}} 
\newcommand{\Q}{\mathbb{Q}}
\newcommand{\Z}{\mathbb{Z}}
\newcommand{\T}{\mathbb{T}}

\newcommand{\s}[1]{\CMcal{#1}}
                  \newcommand{\I}{\mathfrak I}
\newcommand{\W}{\mathcal{W}}
\newcommand{\bb}[1]{\mathscr{#1}}
\newcommand{\rr}[1]{\mathfrak{#1}}
\newcommand{\n}[1]{\mathbb{#1}}

                        \newcommand{\h}{\mathfrak{h}}

\newcommand{\Orb}{\mathtt{Orb}}

\newcommand{\expo}[1]{\,\mathrm{e}^{#1}\,}                 

\newcommand{\dd}{\,\mathrm{d}}
\newcommand{ \ii}{\,\mathrm{i}\,}

\newcommand{\virg}[1]{\lq\lq#1\rq\rq}                \newcommand{\ie}{\textsl{i.\,e.\,}}
\newcommand{\eg}{\textsl{e.\,g.\,}}
\newcommand{\cf}{\textsl{cf}.\,}

\newcommand{\A}{\mathfrak{A}}

\hypersetup{
pdftoolbar=true,        
pdfmenubar=true,        
pdffitwindow=true,     
pdfstartview=true,    
pdftitle={MQP-algebras},    
pdfauthor={Danilo Polo},     
breaklinks=true, 
colorlinks=true,       
linkcolor=purple,         
citecolor=teal, 
urlcolor=blue, 
bookmarksopen=true, 
filecolor=magenta,      
}

\begin{document}

\title[On the K-theory of magnetic algebras: Iwatsuka case ]{On the K-theory of magnetic algebras: Iwatsuka case }

\author[G. De~Nittis]{Giuseppe De Nittis}

\address[G. De~Nittis]{Facultad de Matem\'aticas \& Instituto de F\'{\i}sica,
  Pontificia Universidad Cat\'olica de Chile,
  Santiago, Chile.}
\email{gidenittis@uc.cl}
\author[J. Gomez]{Jaime Gomez}

\address[J. Gomez]{Mathematisch Instituut, University of Leiden, Leiden, Netherlands}
\email{j.a.gomez.ortiz@math.leidenuniv.nl}

\author[D. Polo]{Danilo polo Ojito}

\address[D. Polo]{Department of Physics, Yeshiva University, New York, NY 10016, USA}
\email{danilo.poloojito@yu.edu}

\vspace{2mm}

\date{\today}

\begin{abstract}
In the tight-binding approximation, an Iwatsuka magnetic field is modeled by a function on $\Z^2$ with constant, but distinct values in the two parts of the lattice separated by a straight line of slope $\alpha\in [-\infty,\infty]$.  In this paper,  the $K$-theory of the magnetic $C^*$-algebras generated by an Iwatsuka magnetic field for any possible $\alpha$ is computed. One interesting aspect concerns the analysis of the behavior of the system in the transition from rational to irrational $\alpha$.
It turns out that when $\alpha$ is irrational, the magnetic hull associated with the flux operator forms a Cantor set. On the other hand,  for rational $\alpha$  this set coincides with the two-point compactification of $\Z$. This characterization, along with the use of the Pimsner-Voiculescu exact sequence, is the main ingredient for the computation of the $K$-theory. Once  the $K$-theory is known, with the use of the index theory one can deduce   the bulk-interface correspondence for tight-binding Hamiltonians  subjected to an Iwatsuka magnetic field. Notably, it occurs that the topological quantization of the interface currents remains independent of the slope  $\alpha$.
\medskip


\end{abstract}

\maketitle

\tableofcontents
\section{Introduction}
In the tight-binding approximation, a generalized Iwatsuka magnetic field on the lattice $\Z^2$ is a function $B_\alpha\colon \Z^2\to \R$ which takes constant values $b_\pm$ in the two regions separated by the \emph{interface} described by the straight line $y=\alpha x$ of \emph{slope} $\alpha$. The standard Iwatsuka model is recovered when $\alpha=\pm \infty,$ meaning the dividing line is $x=0$ \cite{Iwa}.  In this paper, we investigate the topological property of the Iwatsuka model for any possible $\alpha\in\overline{\R}:=[-\infty,+\infty]$
with a special interest in the analysis of the behavior of the system in the transition from rational to irrational $\alpha$. 

\medskip

The system of interest will be described in terms of the 
$C^*$-algebra  $\A_\alpha\subset\mathcal{B}(\ell^2(\Z^2))$ generated by magnetic translations $\mathfrak{s}_1$ and $\mathfrak{s}_2$ associated to  the Iwatsuka magnetic field and subjected to the following commutation relation
$$
\big(\mathfrak{s}_1\mathfrak{s}_2\mathfrak{s}_1^*\mathfrak{s}_2^*\psi\big)(n)\;=\;\expo{\ii B_\alpha(n)}\psi(n),\qquad \forall\,\psi\in \ell^2(\Z^2)\;.
$$
We will refer to $\A_\alpha$ as the \emph{Iwatsuka algebra} for short.
In order to compute the $K$-groups of   $\A_\alpha$ the first step  consists is  noticing that the algebra can be decomposed in terms of crossed products, \ie
$$\A_\alpha\;\simeq \;\big(C(\Omega_\alpha)\rtimes \Z\big)\rtimes \Z
$$
where $\Omega_\alpha$, the so called \emph{magnetic hull} \cite{Deni1, Dani}, is a suitable compact Hausdorff space equipped with a $\Z^2$-action $\sigma^*$.   Our first main result provides a precise description of the topology of $\Omega_\alpha$ and all its invariant measures.
\begin{theorem}\label{th_001}
For any $\alpha \in \overline{\R}$ the dynamical system $(\Omega_\alpha,\sigma^*,\Z^2)$ has only three invariant measures $\big\{ \mu_\alpha, \mathbb{P}_+,\mathbb{P}_-\big\}$, up to  a scaling factor. Moreover,
\begin{enumerate}[i.]
\item  $\Omega_\alpha$ is homeomorphic to the two-point compactification of $\Z$ for $\alpha\in \mathbb{Q}\cup\{\pm \infty\}.$
    \item $\Omega_\alpha$ is a Cantor set for $\alpha$ irrational.
\end{enumerate}
\end{theorem}
Sections \ref{sec: hull} and \ref{sec:measures} are devoted to the proof of this result, where an explicit description of the elements of the magnetic hull and the measures is provided. More precisely the arguments for the proof of Theorem \ref{th_001} are contained in Propositions \ref{teo_disc_rat}, \ref{teo cantor}, \ref{Teo. measures} and \ref{prop: invariant measure}.

\medskip

The main implication of Theorem \ref{th_001} is the existence of a set of three invariant traces $\{\mathcal{T}_\alpha,\mathcal{T}_+,\mathcal{T}_-\}$ on $\A_\alpha$   obtained by the composition of the invariant measures with the conditional expectation $E\colon \A_\alpha\to C(\Omega_\alpha)$ (\cf  \cite[Chapter VIII]{Dav}). From the physical side, $\mathcal{T}_\alpha$ is an infinite trace that provides the expectation value of extended observables supported near the interface. The other two traces $\mathcal{T}_\pm$ are finite traces per-unit-volume related with observables supported away from the interface in the two opposite directions.

\medskip

The main ingredient in describing the topology of $\Omega_\alpha$ is the observation that this space is $\Z^2$-homeomorphic to some \emph{subshift space} of  $\Om:=\{b_+,b_-\}^{\Z^2}$. For $\alpha$ irrational it holds true that  $\Omega_\alpha$  is a closed subset of $\Om$ without isolated points. On the other hand, for $\alpha\in \mathbb{Q}\cup\{\pm \infty\}$ the space $\Omega_\alpha$ is a countable set with two accumulation points. It is important to highlight that the characterization of the irrational case is the truly new result	of this work since the rational can be inferred (with a little effort) from the previous work \cite{Deni1}. With this information in mind, the Pimsner-Voicolescu exact sequence \cite{Pim} allows us to compute the $K$-theory $\A_\alpha$. This is  the second main result of this work:
\begin{theorem}\label{Teo: 1.2}
For every $\alpha\in\overline{\R}$ the following holds
   \begin{equation*}
K_0(\A_\alpha)\;=\;\Z^3\;,\qquad K_1(\A_\alpha)\;=\;\begin{cases} \Z^3& \text{if}\; \alpha\in \mathbb{Q}\cup\{\pm\infty\}\\\Z^2&\text{\rm otherwise\;.}
\end{cases}
   \end{equation*}
\end{theorem}
The proof of this Theorem is given in Propositions \ref{teo: k-groups rational} and \ref{k-groups A}, where the generators of these groups are also included. It should be mentioned that the $K$-groups for the case $\alpha=\pm \infty$ were already computed in \cite{Deni1}.

\medskip

The $K$-theory of the magnetic $C^*$-algebras, like the Iwatsuka $C^*$-algebra, has been a widely studied topic in recent years due to its applications to topological effect in transport phenomena like the integer quantum Hall effect (IQHE) \cite{PRO, Kel, Kel1} and the quantized edge currents generated by magnetic interfaces \cite{Dani, Deni1, Kot}. 
In this work, as a concrete application of Theorem \ref{Teo: 1.2}, we study the 
topological quantization of the edge currents along the magnetic barriers induced by $B_\alpha$.
It is known that in the standard Iwatsuka model  ($\alpha=\pm \infty$)  when $b_+\neq b_-$ may exist extended states localized near the interface carrying currents. Moreover,   such currents are quantized, similar to the case of the IQHE,  by an integer which is obtained as the difference of two Chern numbers associated with the constant magnetic fields $b_\pm$. It is worth observing that this is a purely magnetic phenomenon since the magnetic field is what generates the interface while the material remains the same (infinitely extended and without boundaries). From the mathematical point of view, there is quite an extensive literature where the localization, existence, and quantization of interface currents have been proven for the one-particle sector \cite{Dom, Kot,His,Deni1} (see also \cite{Dani, Dro1,Dro2,GV} for non-straight interfaces). By following the $K$-theoretic framework constructed in \cite{Deni1}, we prove the quantization of the interface currents for the generalized Iwatsuka magnetic field $B_\alpha$ for any $\alpha\in\overline{\R}$. We start by constructing an exact sequence of $C^*$-algebras
 
\begin{equation}\label{eq: 1}
    \xymatrix{
 0\ar[r]&\mathfrak{I}_{\alpha}\ar[r]^{} & \A_{\alpha} \ar[r]^{{\rm ev}}& \A_{{\rm bulk}}\ar[r]&0\;, }
\end{equation}
where $\mathfrak{I}_\alpha$ is the interface algebra describing the behavior near the interface, and the bulk algebra $\mathfrak{A}_{\rm bulk}:=\mathfrak{A}_{b_+}\oplus\mathfrak{A}_{b_-}$ contains the information of the asymptotic magnetic field. The full Hamiltonian of the system is a selfadjoint element $\hat{\h}\in \A_\alpha$, and the bulk Hamiltonian is described by $\h:=(\mathfrak{h}_+,\mathfrak{h}_-)={\rm ev}(\hat{\h})\in \mathfrak{A}_{\rm bulk}$. If we assume that there is a compact set $\Delta$ contained in a non-trivial spectral gap of the bulk Hamiltonian $\h$ (\emph{bulk gap assumption}),  then for any $\mu\in\Delta $ the \emph{Fermi projection} 
$$\mathfrak{p}_\mu\;:=\;\big(\mathfrak{p}_{\mu_+},\mathfrak{p}_{\mu_-}\big)\;=\;\big(\,\chi_{(-\infty,\mu]}(\mathfrak{h}_+)\,,\,\chi_{(-\infty,\mu]}(\mathfrak{h}_-)\,\big)\;\in\; \mathfrak{A}_{\rm bulk}$$
defines an element in $[\mathfrak{p}_\mu]\in K_0(\mathfrak{A}_{\rm bulk})$. The bulk magnetic invariants of the system can be written in terms of the Chern numbers of the projections, \ie  $\sigma_{b_\pm}(\mathfrak{p}_\mu):=\frac{e^2}{h}{\rm Ch}(\mathfrak{p}_{\mu_\pm})\in \Z$, where  $e>0$ is the magnitude of the electron charge and $h$ is the Planck's constant.


\medskip

Let $\sigma(\Delta)$ be the conductance associated with the interface current carried by the extended states of $\hat{\h}$  with energy inside of $\Delta.$   Our third main result states that the interface conductance can be written as a difference of the bulk magnetic invariants of the system.


\begin{theorem}[Bulk-interface correspondence]\label{Teo bulk-interface}
    Let the bulk gap assumption be valid and $\mu\in \Delta$. The interface conductance $\sigma(\Delta)$ carried by the states of energy inside   $\Delta$ is given by 
\begin{equation*}
    \sigma(\Delta)\;=\;\sigma_{b_+}(\mathfrak{p}_{\mu})-\sigma_{b_-}(\mathfrak{p}_{\mu})
\end{equation*}
\end{theorem}
The relevance of Theorem \ref{Teo bulk-interface} is its independence of the values of $\alpha\in\overline{\R}$. A precise statement of this result with a detailed proof is presented in Theorem \ref{Teo current}.

\medskip

{The result above offers a not-innocent generalization of the bulk-interface correspondence established in \cite{Deni1, Dom, Kot}. The major innovation here lies in the treatment of interface currents at irrational slope. In fact, in the irrational case, the loss of symmetry along the magnetic interface made complicated the definition of the trace per unit of surface of the interface, and in turn one ends in the challenge of providing a good calculation of the expected value of observables near the interface.
This fact was already noted in the context of geometric defects in materials in \cite[Section 5.1]{Guo}. The present work, not only establishes the existence of a good interface trace. It also shows the uniqueness of such trace (Proposition \ref{prop: interface trace}). Furthermore, this trace provides the correct $1$-cocycle which 
computes the interface currents when paired with a suitable $K$-class. This discovery, along with the techniques introduced here, paves the way for extending well-known results on bulk-edge correspondence to settings where symmetry along the defect is lost. }

   \medskip

{
It is important to point out that there are a few works dealing with materials with irrational straight boundaries. For instance, in the monograph \cite{Tom} is considered the so-called \emph{smooth Toeplitz extension} \cite{Les,JI, JK}  to deal with operators in half-plane spaces with an arbitrary angle. This \virg{smooth approximation} means that the usual half-plane $C^*$-algebra is replaced by a new  $C^*$-algebra obtained by a \virg{smooth} approximation of the half-plane (Toeplitz) projection.
In this setting, the authors construct the trace as the dual of the bulk-trace, and the bulk-edge correspondence is obtained from the Connes-Tom isomorphism and the duality of Chern cocycles.  However, without this smooth approximation, the trace in \cite{Tom} cannot be defined in the same way since the duality is lost. It is worth pointing out that  
our strategy is independent of such an approximation procedure. Moreover, it can be verified that our trace is directly related to the trace in \cite{Tom} as shown in
   Appendix \ref{app smooth}.   Let us end by observing that significant contributions are also contained in \cite{LT} where the coarse geometry is used to establish a version of the bulk-edge correspondence for a broad class of geometries, including those coarsely equivalent to straight-line irrational defects. However, it is difficult, or at least not immediately clear, how this approach could be used to derive the equivalent of Theorem \ref{Teo bulk-interface}. One main problem is related to the fact that  Roe $C*$-algebras are insensitive to the twists generated by magnetic fields (see \eg \cite[Section 2.1]{Mey}). Moreover, the numerical invariants are not well defined in the Roe $C^*$-algebras in view of the lack of a nice trace.}

\medskip

\noindent
{\bf Acknowledgements:}
    The authors would like to cordially thank to M. Moscolari, E. Prodan, and T. Stoiber for several stimulating discussions.  D. P.'s was supported by the U.S. NSF through the grants DMR-1823800 and CMMI-2131760, and by U.S. Army Research Office through contract W911NF-23-1-0127.  G. D.'s research is supported by the grant \emph{Fondecyt Regular - 1230032}. 
 \section{Magnetic algebras}

\subsection{Algebra of the magnetic translations}
In this section, we introduce the magnetic algebras following the presentation in the previous works  \cite{Deni1,Dani}.

\medskip

For two-dimensional systems in the tight-binding approximation, the \emph{magnetic field} orthogonal to the plane is represented by a function $B\colon \Z^2\to \R$ and the corresponding \emph{vector potential} by a function $A_B\colon \Z^2\times \Z^2\to \R$ whose circulation is equal to $B$. Namely, $B(n) =\mathfrak{C}[A_{B}](n)$ for all $n\in\Z^2$, where
\begin{equation*}
    \begin{split}
     \mathfrak{C}[A_{B}](n)\;:=\;&A_{B}(n,n-e_1)+A_{B}(n-e_1,n-e_1-e_2)\\&+A_{B}(n-e_1-e_2,n-e_2)+A_{B}(n-e_2,n)  \;.
    \end{split}
\end{equation*}
Along this work $e_1:=(1,0)$ and $e_2:=(0,1)$ denote the canonical basis of $\Z^2$. The \emph{standard}  vector potential associated with the magnetic field $B$ is given by
   \begin{equation}\label{eq: vector potential}
       A_{B}(n,n-e_j)\;:=\;\delta_{j,1}\left(\delta_{n_2>0}\sum_{m=1}^{n_2}B(n_1,m)-\delta_{n_2<0}\sum_{m=0}^{|n_2|-1}B(n_1,-m)\right)\;,
       \end{equation}
        where $n=(n_1,n_2)\in \Z^2$. 
        \begin{remark}
        When the magnetic field is independent of the $y$-axis one obtains that 
        \[
        A_B(n,n-e_j)\;=\;\delta_{j,1} n_2 B(n)\;,
        \]
         that is, $A_B$ is a \emph{Landau-type} vector potential  \cite[Section 2.2]{Deni3}.   \hfill $\blacktriangleleft$
        \end{remark}
 The magnetic translations $\mathfrak{s}_1$ and $\mathfrak{s}_2$ associated to the magnetic potential $A_{B}$ defined in \eqref{eq: vector potential}, are the unitary operators acting on $\psi\in \ell^2(\Z^2)$ via
  \begin{equation}\label{eq:mag_tras}
  \begin{split}
       (\mathfrak{s}_1\psi)(n)\;&:=\;\expo{\ii A_{B}(n,n-e_1)  }\psi(n-e_1)\;,\\
         (\mathfrak{s}_2\psi)(n)\;&:=\;\psi(n-e_2)\;.
  \end{split}
    \end{equation} 
In this paper we will identify the group $\n{U}(1)$ of the complex numbers of modulus one with the one-dimensional torus $\n{T}$.
Let $f_{B}:\Z^2\to \T$ be the function
\begin{equation}\label{eq:f_B}
f_{B}(n)\;:=\;\expo{\ii B(n)}\;,\qquad n\in\Z^2,
\end{equation}
which provides the  exponential of  the \emph{magnetic flux} through the  unit cell sited in $n$, and define the associated \emph{flux operator} as
\begin{equation*}
    (\mathfrak{f}_{B}\psi)(n)\;:=\; f_{B}(n)\psi(n)\;,\qquad \psi\in \ell^2(\Z^2)\;.
\end{equation*}
A straightforward computation shows that the magnetic translations
 satisfy the condition
\begin{equation}\label{equ 2.4}
    \mathfrak{s}_1\;\mathfrak{s}_2\;=\;\mathfrak{f}_B\;\mathfrak{s}_2\;\mathfrak{s}_1\;.
\end{equation}
\begin{definition}\label{Def: magnetic algebras} The \emph{magnetic} algebra $\A_{B}$ in the standard gauge $A_B$ given in \eqref{eq: vector potential}, is the  sub-$C^*$-algebra of $\mathcal{B}(\ell^2(\Z^2))$ generated by the magnetic translations $\mathfrak{s}_j$, \ie
$$\A_{B}\;:=\;C^*(\mathfrak{s}_1,\mathfrak{s}_2)\;.$$
\end{definition}
\begin{remark}
 As we pointed out before, several examples of magnetic algebras can be found in \cite{Deni1, Dani}. See also the monographs \cite{PRO, Tom}, where disordered magnetic algebras are constructed. \hfill $\blacktriangleleft$  
\end{remark}

For any $\theta=(\theta_1,\theta_2)\in [0,2\pi)^2$ let us consider the unitary operator $\mathfrak{v}_\theta$ which acts on  $\psi\in \ell^2(\Z^2)$  as
\begin{equation}\label{eq: action}
    (\mathfrak{v}_\theta\psi)(n)\;:=\;\expo{-\ii\theta\cdot n}\psi(n)\;
    \end{equation}
    where $\theta\cdot n:=n_1\theta_1+n_2\theta_2$.
These unitary operators define a continuous group action  $\T^2\ni\theta\mapsto \eta_\theta$  
by internal automorphisms
of the magnetic algebra $\A_B$ \cite[Proposition 2.12]{Deni1}. Consequently, one can introduce the infinitesimal generators $\nabla_1$ and $\nabla_2$ as
$$\nabla_1(\mathfrak{a})\;:=\;\lim_{\theta_1\rightarrow0}\frac{\eta_{(\theta_1,0)}(\mathfrak{a})-\mathfrak{a}}{\theta_1}\;,\qquad \nabla_2(\mathfrak{a})\;:=\;\lim_{\theta_2\rightarrow0}\frac{\eta_{(0,\theta_2)}(\mathfrak{a})-\mathfrak{a}}{\theta_2}$$
for  suitable elements $\mathfrak{a}\in \A_B$. Indeed, if $\mathfrak{a}$ lies in the domain of $\nabla:=(\nabla_1,\nabla_2)$ then one gets  that
$$\nabla_i(\mathfrak{a})\;=\;\ii [\mathfrak{n}_j,\mathfrak{a}],\;\qquad j=1,2$$
where $(\mathfrak{n}_j\psi)(n):=n_j\psi(n)$ are the position operators on $\ell^2(\Z^2).$
\medskip

Let $\A_B^0\subset \A_B$ be the dense subalgebra of non-commutative polynomials in the variables $\mathfrak{s}_1$ and $\mathfrak{s}_2$. In this work, we will need to consider  the Banach spaces
\begin{equation}
    \mathcal{C}^k(\A_{B})\;:=\;\overline{\A_{B}^0}^{\|\cdot\|_k},
\end{equation}
where the norm $\|\cdot\|_k$ is given by
$$\|\mathfrak{a}\|_k\;:=\;\sum_{i=0}^k\sum_{a+b=i}\|\nabla_1^a\nabla_2^b \mathfrak{a}\|\;.$$



\subsection{Magnetic hull}

Let $\rr{m}:\ell^\infty(\Z^2)\hookrightarrow \s{B}(\ell^2(\Z^2))$ be the injective $*$-morphism which associates to each $g\in\ell^\infty(\Z^2)$ the multiplication operator $\rr{m}_g$ and consider the $\Z^2$-action defined on $\ell^\infty(\Z^2)$ by
\begin{equation}
    \sigma_\gamma(g)(n)\,:=g(n-\gamma)\;,\qquad n,\gamma\in \Z^2\;.
\end{equation}
For every $g\in \ell^\infty(\Z^2)$ and $\gamma=(\gamma_1,\gamma_2)$ it holds true that 
\begin{equation}\label{equ 2.6}
\sigma_\gamma\left(\mathfrak{m}_{g}\right)\;:=\;    \mathfrak{m}_{\sigma_\gamma(g)}\;=\;(\mathfrak{s}_1)^{\gamma_1}\;(\mathfrak{s}_2)^{\gamma_2}\;\mathfrak{m}_g\;(\mathfrak{s}_2)^{-\gamma_2}\;(\mathfrak{s}_1)^{-\gamma_1}
\end{equation}
 with an innocent abuse of notation in the use of $\sigma_\gamma$.
In view of  \eqref{eq:f_B} and \eqref{equ 2.6} one gets that 
$$
\sigma_\gamma(\mathfrak{f}_{B})\;=\;\mathfrak{m}_{\sigma_\gamma({f}_{B})}\;\in\; \A_{B}\;,\qquad \forall\;\gamma\in \Z^2\;.
$$ 
Hence, the $C^*$-algebra
\begin{equation}\label{eq: dense}
    \mathfrak{F}_{B}\;:=\;C^*\big\{\sigma_\gamma(\mathfrak{f}_{B})\;|\;\gamma\in \Z^2\big\}\;\subset\; \A_{B}
\end{equation}
 is a commutative sub $C^*$-algebra of $\A_{B}$  with unit $\mathfrak{1}$.  Moreover,
by construction, it is invariant under the $\Z^2$-action $\sigma$ implemented by  $\gamma\mapsto\sigma_\gamma$. 
 \begin{definition}
     The magnetic Hull $\Omega_B$ of the algebra $\A_B$ is the Gelfand spectrum of $\mathfrak{F}_B$. Say in other words,  $\Omega_B$ is a compact Hausdorff space such that
     $$ 
\mathfrak{F}_{B}\;\simeq\; C(\Omega_{B})\;.$$
 \end{definition}
\begin{remark}
Given that $\sigma_\gamma(f_B)\in \ell^\infty(\Z^2)$ for all $\gamma\in \Z^2$, it follows that $C(\Omega_B)$ can be identified with a sub-$C^*$-algebra of $\ell^\infty(\Z^2)$. Furthermore, the isomorphisms 
    $$\ell^\infty(\Z^2)\;= \; C_b(\Z^2)\;\simeq \; C(\beta\Z^2)$$
show that there exists a continuous surjective map $\mathtt{r}\colon \beta\Z^2\to \Omega_B$, where $\beta \Z^2$ is the Stone-\v{C}ech compactification of $\Z^2$ and $C_b(\Z^2)$ denotes the algebra of continuous bounded functions. Thus the space $\Omega_B$ can be described in terms of filters on $\Z^2$  (see \cite[Section 2.2]{Dani}
\hfill $\blacktriangleleft$  
\end{remark}

The importance of the magnetic hull relies on the fact that allows us to write any magnetic algebra as an iterated crossed product of $C^*$-algebras \cite{Dav,Wil} or as a twisted crossed-product \cite{PR}. A discussion of this result can be found in \cite[Appendix A]{Deni1}
\begin{proposition}\label{prop: cros}
It holds true that$$\A_{B}\;\simeq  \;\big(\mathfrak{F}_B\rtimes_{\sigma_1}\Z\big)\rtimes_{\sigma_2}\Z\;\simeq\;\mathfrak{F}_B\rtimes_{\sigma,\theta_B}\Z^2$$ as $C^*$-algebras, where the automorphism $\sigma_i$ is given by conjugation for $\mathfrak{s}_i$ with $i=1,2,$ and the $2$-cocycle $\theta_B\colon \Z^2\times \Z^2\to \mathcal{U}(\mathfrak{F}_B)$ is defined via
\begin{equation}\label{cocycle}
    \theta_B(r,s)\;=\;\prod_{n\in \Lambda(r,s)}\sigma_n(\mathfrak{f}_B)^*
\end{equation}
where $\Lambda(r,s)=([r_1,r_1+s_1-1]\times [0,r_2-1])\cap \Z^2$ and $\mathcal{U}(\mathfrak{F}_B)$ is the set of unitary operators on $\mathfrak{F}_B.$
\end{proposition}

Since $\Omega_B$ is the set of characters of $\mathfrak{F}_B$, one can endow $\Omega_B$ with a $\Z^2$-action: For each $\gamma\in \Z^2$ one defines  $\sigma^*_\gamma(\omega)(\mathfrak{g}):=\omega(\sigma_{-\gamma}(\mathfrak{g}))$ with $\omega\in \Omega_{B}$ and $\mathfrak{g}\in \mathfrak{F}_B.$ Let us define the set $\mathtt{Orb}(\omega_0):=\{\sigma_\gamma^*(\omega_0)\,:\,\gamma\in \Z^2\},$ where $\omega_0$ is the character given by the evaluation at the point $(0,0)\in \Z^2$. As a result of \cite[Proposition 2.21]{Deni1}, it turns out that
\begin{equation}\label{eq: dense1}
  \Omega_B\;=\;\overline{\mathtt{Orb}(\omega_0)}\;.  
\end{equation}

\begin{remark}
    The space $\Omega_B$ is a $*$-weak closed subset of the unit ball of the dual space of  $\mathfrak{F}_B$. Furthermore, from its very definition $\mathfrak{F}_B$ is separable, and in turn $\Omega_B$ is metrizable. The metric can be chosen as
\begin{align}
\overline{d}\big(\omega_1,\omega_2)\;:=\;\sum_{i=1}^\infty \dfrac{|\omega_1(\mathfrak{f}_i)-\omega_2(\mathfrak{f}_i)|}{2^i}\;,\qquad \quad\omega_1,\omega_2\in \Omega_{B}.
\end{align}
where $\{\mathfrak{f}_i\}_{i\in \N}$ is any enumeration of the set $\{\sigma_\gamma(\mathfrak{f}_B)\;|\;\gamma\in \Z^2\}$. \hfill $\blacktriangleleft$
\end{remark}

\medskip

Along this work we will focus on magnetic fields whose range is finite, that is, ${\rm Im}(B)=\{b_1,b_2,\dots,b_m\}\subset \R.$ Therefore, it is natural to identify the magnetic hull as a subset of the \emph{full-shift} space 
\[
{\bf \Omega}\;:=\;\{\expo{\ii b_1},\expo{\ii b_2},\dots,\expo{\ii b_m}\}^{\Z^2}\;.
\] 
Let us recall that $\Om$ endowed with the \emph{prodiscrete topology} is a compact metric totally disconnected topological space without isolated points (a Cantor set). A base for this topology is given by the cylinder sets
\begin{equation}\label{base}
    Z(\beta)\;=\;\big\{ x\in \Om\,\colon\, x|_{{\rm supp}(\beta)}=\alpha|_{{\rm supp}(\beta)}\big\}
\end{equation}
parametrised by $\beta\in C_c\big(\Z^2,\{\expo{\ii b_1},\,\dots,\expo{\ii b_m}\}\big)$. There is also a $\Z^2$-action given by
 \begin{equation}
     \sigma'_\gamma(x)(n)\;:=\;x(n+\gamma)\;,\qquad \gamma, n\in \Z^2
 \end{equation}
 for all $x\in \Om$ (see for instance \cite{CeCo12}).
 Explicitly, it is possible to give a metric compatible with the topology as follows. For $x_1,x_2\in \Om$ one sets 
 \begin{align}\label{metric}
    d\big(x_1,x_2\big)\;:=\;\sum_{i=0}^\infty \dfrac{s_i(x_1,x_2)}{2^{i+1}},\;\qquad \;s_i(x_1,x_2)\;=\;\begin{cases}
        0,&\mbox{ if }x_1|_{[-i,i]^2}=x_2|_{[-i,i]^2}\\
1,&\mbox{otherwise},
    \end{cases}
\end{align}
where $[-i,i]^2$ is meant as the finite subset of $\Z^2$ made of points contained in the square centered at zero with sides $2i$ for every $i\in\mathbb{N}$. 
Notice that $f_B$, defined in \eqref{eq:f_B}, lies in $ \Om$ and hence one can consider the associated \emph{sub-shift} $\Omega'_{B}\subset \Om$ generated by $f_{B}$, that is, $\Omega'_{B}:=\overline{\mathtt{Orb}(f_{B})}$. 
\begin{proposition}\label{Prop: iden}
If ${\rm Im}(B)=\{b_1,b_2,\dots,b_m\}$ then there exists a conjugacy between the dynamical systems $(\Omega_{B},\sigma^*,\Z^2)$ and $(\Omega'_{B},\sigma',\Z^2)$, \ie, an homeomorphism $T:\Omega_B\to \Omega_B'$ such that $T(\sigma^*_{n}(\omega))=\sigma'_n(T(\omega))$ for each $\omega\in\Omega_B$ and $n\in\mathbb{Z}^2$.
\end{proposition}
\begin{proof}
Let us choose an enumeration $\{\mathfrak{f}_i\}_{i\in\mathbb{N}}$ of the set  
$\{\sigma_n(\mathfrak{f}_{B})\;|\;n\in \Z^2\}$ so that for every $m\in \N$ the subset  
\[\{\sigma_n(\mathfrak{f}_B)\mid n\in[-m,m]^2\}\;=\;\{\mathfrak{f}_i\mid 1\leq i\leq (2m+1)^2\}
\]
consists of the \virg{first} $(2m+1)^2$ elements. 
Define the map $T\colon \Omega_{B}\to {\bf \Omega}$ by $T(\omega)(n)=\omega(\sigma_{-n}(\mathfrak{f}_{B}))$ for every $n\in \Z^2$ and $\omega\in \Omega_{B}$. Notice that 
 \begin{equation}\label{eq: omega zero}
T(\omega_0)(n)\;=\;\omega_0(\sigma_{-n}(\mathfrak{f}_{B}))\;=\;f_{B}(n),     
 \end{equation}
 so in view of \eqref{eq: dense1} one gets that $T$ is well defined. 
 Let us show that $T$ is continuous. 
  Let $\varepsilon>0$ and $m\in\mathbb{N}$ such that $\frac{1}{2^{m}}\leq\varepsilon$.
  Consider   ${\delta=\frac{1}{2^{(2m+1)^2}}}$.
If $\overline{d}(\omega_1,\omega_2)<\delta$, we obtain that $\omega_1(\mathfrak{f}_i)=\omega_2(\mathfrak{f}_i)$ for every ${1\leq i\leq (2m+1)^2}$ and this leads $$T(\omega_1)(n)\;=\;\omega_1(\sigma_{-n}(\mathfrak{f}_{B}))\;=\;\omega_2(\sigma_{-n}(\mathfrak{f}_{B}))\;=\;T(\omega_2)(n),$$ for every $n\in [-m,m]^2$. Consequently, $d\big(T(\omega_1),T(\omega_2))<\frac{1}{2^{m}}\leq \varepsilon$ and one concludes that $T$ is continuous.
The map $T$ is injective since the relation $T(\omega_1)=T(\omega_2)$ implies that $\omega_1(\sigma_{-n} (\mathfrak{f}_{B}))=\omega_2(\sigma_{-n} (\mathfrak{f}_{B}))$ for every $n\in\mathbb{Z}^2$.
Moreover, for $s\in\mathbb{Z}^2$ one has 
\begin{align*}
    T\big(\sigma^*_n(\omega)\big)(s)\;&=\;\omega\big(\sigma_{-n-s} (\mathfrak{f}_{B})\big)\;=\;T(\omega)(n+s)\;=\;\sigma'_{n}(T(\omega))(s).
\end{align*}
and in turn $T$ is $\Z^2$-equivariant.
Since \eqref{eq: omega zero} says that $T(\omega_0)=f_{B}$, one infers from the continuity and the $\Z^2$-equivariance  that $T(\Omega_{B})=\Omega'_{B}$ \ie, that $T$ is surjective. 
Since $T$ is a continuous bijection between compact and Hausdorff spaces, it turns out
that $T$ is indeed a homeomorphism.
\end{proof}


\section{Structure and $K$-theory of the  Iwatsuka algebras}\label{section iwa}
This section is devoted to the study of magnetic algebras associated with Iwatsuka-like magnetic fields. We compute the magnetic hull and the K-theory for such $C^*$-algebras.

\subsection{Iwatsuka algebras}
Let $\overline{\R}:=[-\infty,+\infty]$ be the extended real line (two points compactification) and $\alpha \in \overline{\R}$. The \emph{Iwatsuka magnetic} field $B_\alpha\colon \Z^2\to \R$ is defined by
\begin{equation}
   {B_{\alpha}}(n)\;:=\; \left\{ \begin{array}{lll}
             b_+ &\quad &  -\alpha n_1+n_2 >  0\\
           b_- &\quad&  \text{otherwise},
             \end{array}
   \right.\qquad    {B_{\pm \infty}}(n)\;:=\; \left\{ \begin{array}{lll}
             b_\mp &\quad &   n_1> 0\\
           b_\pm &\quad &\text{otherwise}. 
             \end{array}   \right.
\end{equation}
where $b_+,b_-\in\R$ so that $b_{+}-b_-\notin 2\pi\Z$. The \emph{Iwatsuka algebra} $\A_{\alpha}$ is the  $C^*$-algebra related to $B_\alpha$ according to the Definition \ref{Def: magnetic algebras}. We shall use the notation  $\mathfrak{f}_\alpha$, $\mathfrak{F}_\alpha$ and $\Omega_\alpha$ for all the objects associated to such $C^*$-algebra. Along this work, we say that $\alpha$ is rational when $\alpha$ is either a rational number or $\pm \infty$, and we say that $\alpha$ is irrational otherwise.
\medskip

\subsection{Iwatsuka magnetic hulls}\label{sec: hull} In this section we characterize in detail the magnetic hull of the Iwatsuka 
algebra with a special focus on the topology. 

\medskip

For every $\alpha\in\R$ let us introduce the subgroup $\n{X}_\alpha$ of $\R$ given by
\[
\n{X}_\alpha\;:=\;\left\{x^\alpha_n:=-\alpha n_1+n_2\;|\; n=(n_1,n_2)\in\Z^2\right\}
\]
The cases $\alpha=\pm\infty$ are similar with $x^{\pm\infty}_n:=\mp n_1$. Observe that $\n{X}_\alpha$ labels the points of $\mathtt{Orb}(\omega_0)$ via the map
\begin{equation}\label{eq: map homeo}
   \mathtt{Orb}(\omega_0)\ni \sigma_n^*(\omega_0)\;\mapsto\;x_n^\alpha \in \n{X}_\alpha
\end{equation}
For $\alpha$ rational this map is a homeomorphism since it is bijective and the group $\n{X}_\alpha$  and the set $\mathtt{Orb}(\omega_0)$ are discrete spaces. Hence $\n{X}_\alpha$ determines the topology of $\Omega_\alpha$ in view of \eqref{eq: dense1}. Nevertheless, for $\alpha$ irrational this correspondence is no longer a homeomorphism although it remains continuous as we shall see in Proposition \ref{prop. surjective}.

\medskip

Since $B_\alpha$ meets the assumptions of Proposition \ref{Prop: iden}, we can use the identification $\omega_0\equiv f_{B}$ in order to work inside the full shift space $\Om=\{\expo{\ii b_+},\expo{\ii b_-}\}.$ In the following, we pursue a detailed description of the elements and the topology of $\Omega_{\alpha}.$  For that, consider $x\in \R$, $n\in \Z^2$ and define the following elements in $\Om$
\begin{equation}\label{omegas}
\begin{split}
\omega_+(n)\;&:=\;\expo{\ii b_+},\qquad  \omega_x(n)\;:=\;\begin{cases}
        \expo{\ii b_+},&\mbox{ if } \;x_n^\alpha-x\;>\;0,\\
        \expo{\ii b_-},&\mbox{ otherwise, }\end{cases}\\
    \omega_-(n)\;&:=\;\expo{\ii b_-},\qquad
\omega_\bullet(n)\;:=\;\begin{cases}
        \expo{\ii b_+},&\mbox{ if } \;x_n^\alpha\;\geq\; 0,\\
        \expo{\ii b_-},&\mbox{ otherwise, }\end{cases}
\end{split}
\end{equation}
\begin{remark}\label{rk:0-bull}
    When $\alpha$ is rational a simple exercise verifies that $\omega_\bullet$ is actually an element in $\mathtt{Orb}(\omega_0)$. Indeed, if $\alpha=p/q$ with $p$ and $q$ relatively prime with $q>0$ and let $s$ and $t$ be integer numbers such that $1=-ps+qt.$  It follows that  $\omega_0\big(n+(s,t)\big)=\omega_\bullet(n)$ because $(-p/q)(n_1+s)+n_2+t>0 $ if and only if $(-p/q) n_1+n_2\geq 0.$ The case $\alpha=\pm\infty$ is similar.       \hfill $\blacktriangleleft$
\end{remark}
\begin{proposition}
    It holds that $\omega_\pm\in \Omega_\alpha$ for all $\alpha\in \overline{\R}$.
\end{proposition}
\begin{proof}
     Notice that for $\alpha\geq 0$ the elements  $\omega_\pm$ can be obtained as a limit of the Cauchy sequences  $\omega^{k}_\pm:=\sigma^*_{\pm k e_2}(\omega_0)$, respectively. The case $\alpha<0$ follows with the same argument by using  the sequences  $\omega^{k}_\pm:=\sigma^*_{\pm k e_1}(\omega_0)$ and so one gets that $\omega_\pm\in \Omega_\alpha$ for each  $\alpha$.
\end{proof}
Now we characterize the magnetic hull of the Iwatsuka algebra for any $\alpha$. We start with the rational case which is just an adaptation of \cite[Example 2.24]{Deni1} which concerns the cases $\alpha=\pm\infty$. Let us point out that the symbol $\sqcup$ is used in the following to denote the \emph{disjoint} union. We will also use the short notation $\Z_q:=\Z/(q\Z)$ and the convention $\Z_0:=\{\ast\}$ a singleton.

\begin{proposition}\label{teo_disc_rat}
       Let $\alpha=p/q$ with $p$ and $q$ relatively primes and $q\geq0$. Then, $\Omega_\alpha$ is a countable space and  agrees with the two-point compactification of $\Z\times\Z_q$, \ie
 \begin{equation}
       \Omega_\alpha\;=\;\mathtt{Orb}(\omega_\bullet)\sqcup\big\{ \omega_+,\omega_-\big\}\;\simeq\; (\Z\times\Z_q)\sqcup\big\{\pm \infty\big\}
   \end{equation}
\end{proposition}
\begin{proof}
 The map introduced in \eqref{eq: map homeo} and  Remark \ref{rk:0-bull} provide the following $\Z^2$-equivariant homeomorphism
\[
\mathtt{Orb}(\omega_\bullet)\;=\;\mathtt{Orb}(\omega_0)\;\simeq\;\n{X}_\alpha\;.
\]
For $\alpha=\pm\infty$  or $\alpha=0$ (in both cases one can assume $q=0$) one has that $\mathbb{X}_\alpha=\Z$ and this case reduces to \cite[Example 2.24]{Deni1}. If $\alpha=p/q$ with  $q>0$, then 
$\mathbb{X}_\alpha \simeq \Z\times \mathbb{Z}_q$ as a group via the map
    $$\mathbb{X}_\alpha\ni - (p/q)n_1+n_2\;\mapsto\; \big(-pn_1'+n_2,[r]\big)$$
    where $n_1=qn_1'+r$ with $0\leq r\leq q-1$.
Moreover, 
the $\Z^2$-action on $\mathbb{X}_\alpha$ becomes the standard  $\Z$-action by translation on the factor $\Z$. Thus one also concludes this case with \cite[Example 2.24]{Deni1} and this yields that $\Omega_\alpha$ is homeomorphic to the two-point compactification of $\Z\times \mathbb{Z}_q$. 
\end{proof}

\medskip

Let us  observe that $\Z\times \mathbb{Z}_q$  becomes \virg{larger} as $q$ grows. This   gives us a clue as to why $\Omega_\alpha$ becomes uncountable for  $\alpha$ irrational  as described in the next result.

\begin{proposition} \label{teo cantor}
    The compact Hausdorff space $\Omega_\alpha$ is a Cantor set when $\alpha\in\R\setminus\Q$. An explicit description is 
   \begin{equation}\label{irrational_cantor}
       \Omega_\alpha\;=\;\mathtt{Orb}(\omega_\bullet)\sqcup\big\{\omega_x\big\}_{x\in \R}\sqcup\big\{ \omega_+,\omega_-\big\}\;.
   \end{equation}
where the topology is the subspace topology induced from the full-shift space $\Om.$

\end{proposition}
\begin{proof}
Since $\Omega_\alpha$ is a subset of the Cantor set $\Om$, then to conclude the proof it is enough to check the equality of sets and that all the elements in $\Omega_\alpha$  are limit points.  For that, let us first show that $\omega_x\in \Omega_\alpha$ for every $x\in\mathbb{R}$. If $x$ lies in the dense subgroup $\n{X}_\alpha$ of $\R$ it follows from its very definition that $x=x_n^\alpha$ for some $n\in \Z^2$, and this leads to $\omega_x=\sigma_n^*(\omega_0)\in \Omega_\alpha$. Now if $x\notin \n{X}_\alpha,$ there exists a sequence $\{x_j\}_{j\in\N}\subset \n{X}_\alpha$  such that $\lim_{j\to \infty}x_j=x$.  For $\epsilon>0$, consider $M\in \N$ so that $\epsilon>2^{-(M+1)}$ and define
\begin{equation}\label{eq: delta}
        \delta_M\;:=\;\min_{m\in [-M,M]^2}\big|x_m^\alpha-x\big|\;>\;0
\end{equation}
One has  $|x_j-x|<{\delta_M}$ for $j$ big enough. 
Let $n(j)\in\Z^2$ be such that $x_j\equiv x_{n(j)}^\alpha$.
We claim that
\begin{align}\label{Eq-m1}
    \sigma^*_{n(j)}(\omega_0)(m)\;=\;\omega_x (m)\;,\qquad  \forall \,m\in[-M,M]^2
\end{align}
for $j$ sufficiently large.
From its very definition, it follows that \eqref{Eq-m1} is fulfilled whenever the numbers $x_{m-n(j)}^\alpha$ and $x_m^\alpha-x$ have the same sign. A direct computation provides 
\begin{align}\label{eq_MM}
    \big|x_{m-n(j)}^\alpha-(x_m^\alpha - x)\big|\;&=\;\big|x_{n(j)}^\alpha-x\big|\;=\;|x_j-x|\;<\;{\delta_M}\;
\end{align}
for $j$ big enough and for every $m \in[-M,M]^2$. Let us now suppose that there exists a $m'\in[-M,M]^2$ such that $x_ {m'-n(j)}^\alpha$ and $x_{m'}^\alpha - x$   have different signs. Then
\[
\big|x_ {m'-n(j)}^\alpha-( x_{m'}^\alpha- x)\big|\;=\;|x_ {m'-n(j)}^\alpha|+|x_{m'}^\alpha - x|\;>\;\delta_M
\]
which is in contradiction with \eqref{eq_MM}. Therefore \eqref{Eq-m1} is proved and one gets 
$$
d\big(\sigma^*_{n(j)}(\omega_0),\omega_x\big)\;<\;\frac{1}{2^{M+1}}\;<\;\epsilon$$
for $j$ sufficiently large. This  yields  $\omega_x\in \overline{\mathtt{Orb}(\omega_0)}=\Omega_\alpha$ for every $x\in\R$. To verify that $\omega_\bullet\in \Omega_\alpha$, it is enough to check that the sequence $\{\omega_{-j^{-1}}\}_{j\in \N}\subset\Omega_\alpha$  converges to $\omega_\bullet$ when $j$ goes to infinity. As a consequence, one also has that $\Orb(\omega_\bullet)\subset \Omega_\alpha$.
 To sum up, one infers that $\{\omega_x\}_{x\in \R}\cup\{\omega_+,\omega_-\}\cup\Orb(\omega_\bullet)\subseteq \Omega_\alpha$.
Now let us verify the other inclusion.   Let $\omega\in\Omega_\alpha\setminus\{\omega_+,\omega_-\}$ and choose a sequence $\{n(j)\}_{j\in\mathbb{N}}\subseteq\mathbb{Z}^2$ such that $\{\sigma^*_{n(j)}(\omega_0)\}_{j\in\mathbb{N}}$ converges to $\omega$. Since $\omega$ is not constant there exists $r\in \N$ so that $\omega(0,r+1)\neq \omega(0,r)$. Consider now $M\in \N$ such that $M>r+1$, so there is $j_M\in \N$ such that  $\omega|_{[-M,M]^2}=\sigma^*_{n(j)}(\omega_0)|_{[-M,M]^2}$ for $j\geq j_M$. Since $(0,r)$ and $(0,r+1)$ belong to $[-M,M]^2$ then the definition of $\sigma^*_{n(j)}(\omega_0)$ with the above equality show that  the numbers $x_{(0,r)-n(j)}^\alpha$  and $x^\alpha_{(0,r+1)-n(j)}$ have different signs, that is,  $x_{(0,r+1)-n(j)}^\alpha>0$ and $x_{(0,r)-n(j)}^\alpha\leq 0$. 
Therefore one concludes that

\[
x_{(0,r)}^\alpha\;\leq\; x_{n(j)}^\alpha\;<\; x_{(0,r+1)}^\alpha
\]
So $\big\{x_{n(j)}^\alpha\big\}_{j\geq j_M}$ is a bounded sequence in $\R$. By the Bolzano-Weierstrass Theorem, we can assume that $x_{n(j)}^\alpha$ is a monotone sequence that converges to some $x\in \R$ as $j\to\infty$. Observe that when $x\notin\n{X}_\alpha$ one can proceed as before and check that the sequence $\sigma_{n(j)}^* (\omega_0)$ converges to $\omega_x$. On the other hand, if $x=x_n^\alpha\in \n{X}_\alpha$ and the sequence $x_{n(j)}^\alpha$ is increasing or decreasing then $\sigma_{n(j)}^* (\omega_0)$ converges either to  $\sigma_{n}^* (\omega_0)$ or $\sigma_{n}^* (\omega_\bullet)$, respectively. By uniqueness of the limit $\omega\in \Omega_\alpha$ in all the cases, one can conclude the other inclusion and \eqref{irrational_cantor}. Finally, $\Omega_\alpha$ is a Cantor set since it is a closed subset of $\Om$ without isolated points. \end{proof}

\begin{remark}
When $\alpha$ is irrational a base for the topology of $\Omega_\alpha$    is given by the collection $\mathfrak{B}$ of sets of the form $Z(\beta)\cap \Omega_\alpha$ 
     where the $Z(\beta)$'s are defined in \eqref{base}. In terms of the metric \eqref{metric}  the set
     $$B_r(\omega)\;=\;\big\{\omega'\in \Omega_\alpha\;|\; \omega'|_{[-i,i]^2}=\omega|_{[-i,i]^2}\big\}$$ 
     describes
     the open ball    of radius $r=1/2^{i+1}$  and center in $\omega\in \Omega_\alpha$.
     \hfill $\blacktriangleleft$
\end{remark}

\subsection{Invariant measures on the magnetic hulls}\label{sec:measures} In this section we compute all the invariant measures of the magnetic hull of the Iwatsuka algebra.

\medskip

   There is a natural action of $\n{X}_\alpha$ on $\R$  by translations:
\begin{equation}\label{gamma action}
    \zeta_\gamma(x)\;:=\;x+\gamma\;,\qquad \gamma\in \n{X}_\alpha
\end{equation}
   This action could be extended to $\overline{\R}$ by letting the points $\pm \infty$ invariants. It turns out that the map given in \eqref{eq: map homeo} can be continuously extend to a map $\pi:\Omega_\alpha\to \overline{\R}$ which is equivariant with respect to the $\Z^2$-action $\sigma^*$ on $\Omega_\alpha$ and the $\n{X}_\alpha$-action $\zeta$ on $\overline{\R}$. This fact is almost immediate for $\alpha\in\n{Q}\cup\{\pm\infty\}$   in view of the discreteness of  $\Omega_\alpha$ proved in Proposition \ref{teo_disc_rat}.
   The irrational case, based on the description of $\Omega_\alpha$   provided in Proposition \ref{teo cantor}, is more technical:
\begin{proposition}\label{prop. surjective} Let $\alpha\in\R\setminus\Q$.  
 The map $\pi\colon \Omega_\alpha\to \overline{\R}$ given by
\begin{equation}
    \begin{split} \pi(\omega_x)\;=\;x, \qquad \pi\big(\sigma^*_n(\omega_\bullet)\big)\;=\;x_n^\alpha,\qquad \pi(\omega_\pm)\;=\;\pm\infty\
    \end{split}
\end{equation}
 is continuous, surjective, and meets the following equivariance condition
\begin{equation}\label{eqivariance}
    \pi\big(\sigma_n^*(\omega)\big)\;=\;\zeta_{x_n^\alpha}\big(\pi(\omega)\big)\;=\;\pi(\omega)+x_n^\alpha\;,\qquad n\in \Z^2,\;\omega\in \Omega_\alpha\;.
\end{equation}
\end{proposition}

\begin{proof}
It is a matter of straightforward computations to check that $\pi$ is continuous on $\omega_\pm$. So let us focus on the remaining points. The idea is to use sequential continuity which is equivalent on these spaces. Let us start with a convergent sequence of the form $\omega_{x_j}\to \omega$ with $x_j\in \R$ and $j\in \N$. If $\omega=\omega_x$ for some $x\notin \n{X}_\alpha$ consider $\delta_M>0$ defined in   \eqref{eq: delta} for some $M\in \N$. In view of $\n{X}_\alpha$ is dense in $\R$ one gets that $\delta_M\to 0$ as $M$ goes to infinity. 
Moreover, by the definition of the convergence of $\omega_{x_j}$, for any $M$ there exists $j$ big enough such that the numbers $x_n^\alpha-x_j$ and  $x_n^\alpha-x$ have the same sign for every $n$ in the square $[-M,M]^2$.
As a consequence, $|x_j-x|<\delta_M$ for $j$ big enough and this leads to $\pi(\omega_{x_j})\to \pi(\omega_x).$  Meanwhile if the limit point $\omega$ is either $\omega_x$ with $x=x_m^\alpha\in \n{X}_\alpha$ or $\omega=\sigma_m^*(\omega_\bullet),$ we can proceed as before by taking in this case 
   $$\hat{\delta}_M\;=\;\min_{n\in [-M,M]^2\setminus\{m\}}\big|x_n^\alpha-x\big|\;>\;0$$
and arrives in $|x_j-x|<\hat{\delta}_M\to 0$. 
Now if the sequence is contained in $\mathtt{Orb}(\omega_\bullet)$, that is, a sequence of the form $\sigma_{n(j)}^*(\omega_\bullet)$ with $n(j)\in \Z^2$, then the same analysis also applies for this case.
To sum up, if $\{\omega_j\}_{j\in \N}$ is sequence in $\Omega_\alpha$ that converges to $\omega$, then $\pi(\omega_j)\to \pi(\omega)$ as $j\to \infty.$ This concludes the continuity of $\pi.$\\
The surjectivity comes from the definition of $\pi.$ Finally, one has the relation
$$\pi\big(\sigma_n^*(\omega_x)\big))\;=\;\pi(\omega_{x+x_n^\alpha})\;=\;x+x_n^\alpha\;=\;\zeta_{x_n^\alpha}\big(\pi(\omega_x)\big)$$
which implies the equivariance \eqref{eqivariance} since the other cases follow from its very definition.
\end{proof}

Let us define the \emph{interface Hull} as the open invariant subset $\Omega_\alpha^\circ:=\Omega_\alpha\setminus\{\omega_\pm\}$.  Proposition \ref{teo_disc_rat} implies that for $\alpha\in\Q\cup\{\pm\infty\}$ one has the isomorphism $C_0(\Omega_\alpha^\circ)\simeq C_0(\n{X}_\alpha)$ of $C^*$-algebras induced by (the inverse of) the map \eqref{eq: map homeo}. When $\alpha$ is irrational Proposition \ref{prop. surjective} provides the following embedding of $C^*$-algebras
\begin{equation}\label{eq: embedding}
    \xymatrix{
 & C_0(\R)\ar[d]_{\hat{\pi}}\ar@{^{(}->}[r]^{i} & C(\overline{\R})\ar[d]^{\hat{\pi}} \\
&C_0(\Omega_\alpha^\circ)\ar@{^{(}->}[r]^{i} &  C(\Omega_\alpha) }
\end{equation}
where $\hat{\pi}(f)(\omega):=f(\pi(\omega))$ with $f\in C(\overline{\R})$ and $\omega\in \Omega_\alpha.$ These maps will play an important role in the computations of the invariant measures on $\Omega_\alpha.$

\begin{remark}
    The real line $\R$ could be considered as a \virg{smooth} interface hull. In view of it generates a suitable smooth edge algebra (\cf \cite{JI, JK, RX}  and \cite[Chapter 4]{Tom}). We will present a discussion of this algebra in Appendix \ref{app smooth}. \hfill $\blacktriangleleft$
\end{remark}

\begin{proposition}\label{Teo. measures}
    For any $\alpha\in \overline{\R}$ the dynamical system $(\Omega_{\alpha},\Z^2,\sigma^*)$ has only two ergodic probability measures $\mathbb{P}_{+}$ and $\mathbb{P}_{-}$ specified by $\mathbb{P}_{\pm}(\{\omega_\pm\})=1$. 
\end{proposition}
\begin{proof}
Since $\omega_\pm$ are invariant points then it is not difficult to see that $\mathbb{P}_{\pm}$ are ergodic measures on $\Omega_\alpha$.  Now let us suppose that there exists another ergodic measure $\mu$ on $\Omega_\alpha$. Then the support of $\mu$ would be contained in the invariant subset $\Omega_\alpha^\circ$ and  $\mu(\Omega_\alpha^\circ)=1.$ For $\alpha$ rational the pushforward measure of $\mu$ through the map \eqref{eq: map homeo} induces the counting measure on the discrete subgroup $\n{X}_\alpha$ and this yields a contradiction.
On the other hand, for $\alpha$ irrational one knows $\pi(\Omega_\alpha^\circ)=\R$, so one can consider the pushforward $\mu$ measure induced by $\pi$, which is a probability measure on $\R$ invariant under the action of the subgroup $\n{X}_\alpha$ by Proposition \ref{prop. surjective}. Moreover, this group is dense on $\R$ so $\mu$ is, up to a positive constant, the Haar measure on $\R$. The above is a contradiction and this completes the proof.
\end{proof}

Arguing as in the proof of Proposition \ref{Teo. measures}, one can check that there is a unique invariant regular Borel measure $\mu_\alpha$, up to a positive scaling factor, supported in $\Omega_\alpha^\circ$. Indeed,  for $\alpha=p/q$   one can define $\mu_\alpha$ as the scale counting measure on $\Omega_\alpha^\circ\simeq \n{X}_\alpha$
\begin{equation}\label{eq:def_c}
    \mu_\alpha(A)\;=\;c_\alpha\; \sharp(A)\;,\qquad c_\alpha\;:=\frac{1}{\sqrt{p^2+q^2}}
\end{equation}
where $\sharp(A)$ stands for the cardinality of $A$ and  $c_\alpha$ is a convenient normalization constant whose role will  be clarified in Remark \ref{remark 1}.
For $\alpha=\pm\infty$ (with $p=\pm1$ and $q=0$) the subgroups $\n{X}_{\pm\infty}$ coincide with $\Z$ and $\mu_{\pm\infty}$ are just the counting measures (\ie $c_{\pm\infty}=1$). It is worth to observe that the measure $\mu_\alpha$ defined as above has the meaning of \emph{measure-per-unit length} along the interface. For $\alpha\in\R\setminus\Q$ a similar measure can be defined as follows. For each borelian $A\subset \Omega_\alpha^\circ$ one has that
$$\mu_\alpha(A)\;:=\;\lambda\big(\pi(A)\big)$$
where $\lambda$ is the usual Lebesgue measure on $\R$. Observe that this measure is well-defined since $\pi\colon \Omega_\alpha\to \overline{\R}$ meets all the conditions of \cite[2.2.13]{FED}, hence $\pi(A)$ is Lebesgue measurable.
\begin{lemma}
Let $\alpha\in\R\setminus\Q$. Then, $\mu_\alpha$ is a well-defined invariant regular measure on $\Omega_\alpha^\circ$.   
\end{lemma}

\proof
  One has that
\begin{equation}
\mu_\alpha(\emptyset)\;=\;\lambda(\pi(\emptyset))\;=\;\lambda(\emptyset)\;=\;0\;. 
\end{equation}
From the definition of $\pi$, it follows
$$
\n{X}_\alpha\;=\;\left.\big\{ x\in \R\;\right|\;\exists\;\omega_1,\omega_2\in \Omega_\alpha^\circ, \;\omega_1\neq \omega_2\; {\rm and}\; \pi(\omega_1)=\pi(\omega_2)=x\big\}$$ 
and moreover, $\lambda(\n{X}_\alpha)=0$ since this set is countable. Now let $A$ and $B$ two disjoint Borel sets in $\Omega_\alpha^\circ$.  Notice that $\pi(A)$ and $\pi(B)$ may not be disjoint but $$\mu_\alpha(A\cap B)\;=\;\lambda\big(\pi(A\cap B)\big)\;\leq \;\lambda(\n{X}_\alpha)\;=\;0.$$  The latter leads to $   \mu_\alpha(A\cup B)=\mu_\alpha(A)+\mu_\alpha(B).$ To verify the $\sigma$-additive, suppose $\{ A_i\}$ is a non-decreasing monotone family of Borel sets in $\Omega_\alpha^\circ.$ Then $\{ \pi(A_i)\}$ is a non-decreasing monotone family of Borel sets in $\R$. So one obtains
\begin{equation*}
    \begin{split}
        \mu_\alpha\big(\bigcup_i A_i\big)\;&=\;\lambda\left(\pi\big(\bigcup_i  A_i\big)\right)\;=\;\lambda\left(\bigcup_i\pi(  A_i)\right)\;=\;\lim_i\lambda\left(\pi(A_i)\right)\\&=\;\lim_i\mu_\alpha(A_i)
    \end{split}
\end{equation*}
In conclusion, $\mu_\alpha$ is a Borel measure on $\Omega_\alpha^\circ$ which is $\sigma^*$-invariant by Proposition \ref{prop. surjective}.
\qed

\medskip

Assume now that $\nu$ is another invariant measure on $\Omega_\alpha^\circ$. Then from the arguments used in the proof of Proposition \ref{Teo. measures}, the pushforward measure on $\R$ induced by $\pi$ is $k\lambda$ for some $k>0.$ As a conclusion, $\nu=k\mu_\alpha.$  In other words we showed that $\mu_\alpha$ is unique up to a multiplication constant. This ambiguity can be fixed by a suitable normalization.
\begin{remark}[Normalization]\label{rk:norm}
   In order to fix the normalization for the invariant measure we define $\mu_\alpha$ in such a way that $\mu_\alpha(I_1)=1$ where $I_1:=\pi^{-1}([0,1])$ \hfill $\blacktriangleleft$
\end{remark}

\medskip

Summing up one lands on the following result.
\begin{proposition}\label{prop: invariant measure}
    There is a unique regular Borel invariant measure $\mu_\alpha$ on $\Omega_\alpha$ supported in $\Omega_\alpha^\circ$ with normalization given as in Remark \ref{rk:norm}.
\end{proposition}

\medskip

We will call the ergodic measures $\mathbb{P}_\pm$  the \emph{bulk measures} of $\Omega_\alpha$ and $\mu_\alpha$  the \emph{interface measure.} In Section \ref{section interface and bulk} we will see that these measures provide unique $\Z^2$-invariant traces on the bulk algebra and the interface algebra.

\subsection{Bulk and interface algebras}\label{section interface and bulk}
Following \cite{Deni1}, when one considers a full system with asymptotic  constant magnetic fields $b_+$ and $b_-$, it results appropriate   to define the \emph{bulk algebra} as
$$\A_{{\rm bulk}}\;:=\;\A_{b_{-}}\oplus\A_{b_+},$$
where $\A_{b_{+}}$ and $\A_{b_-}$ are the magnetic $C^*$-algebras associated to the constant magnetic fields of strength $b_{+}$ and $b_-$, respectively \cite[Example 2.10]{Deni1}. In other words, the bulk algebra contains asymptotic information about the system. 
\medskip

Since the sets $\{\omega_+\}$ and $\{\omega_-\}$ are closed  invariants subset of $\Omega_\alpha$, then \cite[Proposition 3.11]{Deni1} shows that the map ${\rm ev}\colon \A_{\alpha}\rightarrow\A_{{\rm bulk}}$ defined as

\begin{equation}\label{sec 3.1}
\begin{split}
     {\rm ev}(\mathfrak{s}_1)\;:&=\;(\mathfrak{s}_{b_{-},1},\,\mathfrak{s}_{b_+,1})\\
       {\rm  ev}(\mathfrak{s}_2)\;:&=\;(\mathfrak{s}_{b_{-},2},\,\mathfrak{s}_{b_+,2})\\
        {\rm ev}(\mathfrak{f}_\alpha)\;:&=\;\big(\,\expo{\ii b_{-}}\mathfrak{1},\,\expo{\ii b_+}\mathfrak{1}\big)
\end{split}
\end{equation}
is an evaluation homomorphism according to \cite[Definition 3.1]{Deni1}. We denote by $\mathfrak{I}_{\alpha}$  the kernel of ${\rm ev}$ and we call this ideal the \emph{interface algebra}. By definition, one lands in the following exact sequence
\begin{equation}\label{seq a}
    \xymatrix{
 0\ar[r]&\mathfrak{I}_{\alpha}\ar[r]^{i} & \A_{\alpha} \ar[r]^{{\rm ev}}& \A_{{\rm bulk}}\ar[r]&0\;. }
\end{equation}

\medskip
Proposition \ref{prop: cros} gives a description of $\mathfrak{I}_\alpha$ in terms of crossed product:
\begin{equation}\label{eq: interface}
\mathfrak{I}_\alpha\;\simeq \;\big(C_0(\Omega_\alpha^\circ)\rtimes_{\sigma_1}\Z\big)\rtimes_{\sigma_2}\Z\;\simeq\;C_0(\Omega_\alpha^\circ)\rtimes_{\sigma, \theta_\alpha}\Z^2
\end{equation}
 In agreement with \cite[Proposition 2.4]{Dani}, given $n\in \Z^2$ we introduce the following elements in $\A_\alpha$
\begin{equation}\label{eq: projections}
    \begin{split}
        \mathfrak{r}_{n}\;&:=\;{\rm sign}(\alpha)(\expo{\ii b_+}-\expo{\ii b_{-}})^{-1}\big(\sigma_{(n+e_1)}(\mathfrak{f}_\alpha)-\sigma_{n}(\mathfrak{f}_\alpha)\big),\\
        \mathfrak{l}_{n}\;&:=\;(\expo{\ii b_-}-\expo{\ii b_{+}})^{-1}\big(\sigma_{(n+e_2)}(\mathfrak{f}_\alpha)-\sigma_{n}(\mathfrak{f}_\alpha)\big),\\
\mathfrak{q}_{n}\;&:=\;(\expo{\ii b_+}-\expo{\ii b_{-}})^{-1}\big(\sigma_{n}(\mathfrak{f}_\alpha)-\expo{\ii b_{-}}\mathfrak{1}\big),\\
\mathfrak{q}_{n}^{\bot}\;&:=\;\mathfrak{1}-\mathfrak{q}_n^\alpha\;.
    \end{split}
\end{equation}
Here ${\rm sign}(\cdot)$ is the standard sign function. It is not difficult to see that all these elements are projections in $\mathfrak{F}_\alpha$ and the following relations hold
$${\rm sign}(\alpha)\big(\mathfrak{q}_{n+e_1}-\mathfrak{q}_{n}\big)\;=\;\mathfrak{r}_n\;,\qquad \mathfrak{q}_{n}-\mathfrak{q}_{n+e_2}\;=\;\mathfrak{l}_n\;.
$$
For any point $(q,p)\in\Z^2$ let
\[
\Z_{q,p}\;:=\;\big\{ m\cdot(q,p):=(mq,mp)\;|\; m\in \Z\big\}
\]
be the subgroup of $\Z^2$ of the all  multiples of $(q,p)$. 

\medskip

From this part on, let us identify $\I_\alpha$ with its crossed product structure, in terms of $\Omega_\alpha^\circ$, given in \eqref{eq: interface}. Consider the dense $*$-subalgebra  of $\I_\alpha$  defined by
\begin{equation}
    \mathfrak{I}_\alpha^0\;:=\;\left.\left\{ \sum_{n\in \Lambda} \mathfrak{a}_{n}\mathfrak{s}^{n}\;\right|\; \mathfrak{a}_{n}\in C_c(\Omega_\alpha^\circ)\;,\;\;\; \Lambda\in \bb{P}_f(\Z^2)\right\}
\end{equation}
where $\mathfrak{s}^{n}:=\mathfrak{s}_1^{n_1}\mathfrak{s}_2^{n_2}$
for any $n=(n_1,n_2)\in\Z^2$ and $\bb{P}_f(\Z^2)$ is the family  subsets of $\Z^2$ with  finite cardinality.

\begin{theorem}\label{teo: interface}
 Let $\mathfrak{K}$ be the $C^*$algebra of compact operators. For $\alpha$ rational one has
    \begin{equation}\label{interfae rational}
        \I_\alpha\;\simeq\;C(\T)\otimes \mathfrak{K}\;,
    \end{equation}
while for $\alpha$ irrational
 \begin{equation}\label{eq: interface projected}
\I_\alpha\;\simeq\;\big(\mathfrak{l}_0\I_\alpha\mathfrak{l}_0\big)\otimes \mathfrak{K}\;.
    \end{equation}
\end{theorem}
\begin{proof}  Let us start with  $\alpha=p/q$ with $p,q\in \Z$ relatively primes. One has the following isomorphisms
  \begin{equation*}
  \begin{split}
       \I_\alpha\;&\simeq \; \big(C_0(\Omega_\alpha^\circ)\rtimes_{\sigma,\theta_\alpha}\Z_{q,p}\big)\rtimes_{\sigma\circ c, \theta'_\alpha}\Z^2/\Z_{q,p}\\
       &\simeq \; \big(C_0(\Omega_\alpha^\circ)\rtimes_{\sigma,\theta_\alpha}\Z_{q,p}\big)\rtimes_{\sigma, \theta_\alpha}\Z_{-p,q}\;.
  \end{split}
\end{equation*}
In the first line we used \cite[Theorem 4.1]{PR}, where $c\colon \Z^2/\Z_{q,p}\to \Z^2$ is a Borel section and $\theta'_\alpha$ is the induced $2$-cocycle in the quotient by $\theta_\alpha$. The second line is a consequence of the group isomorphism $\Z^2/\Z_{q,p}\simeq \Z_{-p,q}$ induced by the surjective homomorphism 
$$\Z^2\ni (n_1,n_2)\;\mapsto\; (-pn_1+qn_2)\cdot (-p,q)
$$ 
Now observe that $\mathfrak{f}_\alpha$ is invariant under the action of the subgroup $\Z_{q,p}\subset \Z^2$, so one can proceed with the Bloch-Floquet transform as in \cite[Proposition 4.7]{Deni1} and arrive in \eqref{interfae rational}.  The case $\alpha=\pm\infty$ is just the same as in \cite[Proposition 4.7]{Deni1}.
Coming back to the irrational case let $W\colon \ell^2(\Z^2)\to \ell^2\big(\Z,\ell^2(\Z^2)\big)\simeq \ell^2(\Z)\otimes \ell^2(\Z^2)$ be the unitary operator acting via
\begin{equation}
W\psi\;:=\;\big(\mathfrak{l}_{0}\mathfrak{s}_2^{-j}\psi\big)_{j\in \Z}\;.
\end{equation}
Here $\mathfrak{l}_n$ is the projection defined in \eqref{eq: projections}. A simple computation verifies that for any operator $\mathfrak{a}$ in $\ell^2(\Z^2)$ it is true that
\begin{equation}\label{eq: matrix}
W\mathfrak{a}W^*\;=\;\big(\mathfrak{l}_0 \mathfrak{s}_2^{-i}\mathfrak{a}\mathfrak{s}_2^{j}\mathfrak{l}_0\big)_{i,j}\;.   
\end{equation}
We claim that 
$$W\mathfrak{I}_\alpha W^*\;=\;\big(\mathfrak{l}_0\mathfrak{I}_\alpha\mathfrak{l}_0\big)\otimes \mathfrak{K}(\ell^2(\Z))\;.$$
Indeed, let $\mathfrak{a}\in \mathfrak{I}_\alpha$ be the image of the element $\hat{\mathfrak{a}}=\sum_{n\in \Z^2}\mathfrak{a}_n\mathfrak{s}^{n}$ under the isomorphism \eqref{eq: interface}, where $\mathfrak{a}_n\in C_c(\Omega_\alpha^\circ)$ and the sum has a finite number of non-zero elements. Then the matrix associated with $W\mathfrak{a}W^*$ in the r.h.s. in equation \eqref{eq: matrix} has only a finite number of nonzero entries. Thus, $W\mathfrak{a}W^*\in \big(\mathfrak{l}_0\mathfrak{I}_\alpha\mathfrak{l}_0\big)\otimes \mathfrak{K}(\ell^2(\Z)$ and by the density of $\mathfrak{I}_\alpha^0$ in $\mathfrak{I}_\alpha$ one gets the first containing. For the reverse inclusion, consider a matrix $\big(\mathfrak{a}_{ij}\big)_{i,j}$ with a finite number of nonzero entries in $\mathfrak{l}_0\mathfrak{I}_\alpha \mathfrak{l}_0$. Then if one defines $\mathfrak{a}=\sum_{i,j}\mathfrak{s}_2^i\mathfrak{a}_{ij}\mathfrak{s}_2^{-j}$  it turns out that $\mathfrak{a}\in \mathfrak{I}_\alpha$ and $W\mathfrak{a}W^*=\big(\mathfrak{a}_{ij}\big)_{i,j}$. Then the claim follows and this shows \eqref{eq: interface projected}.
\end{proof}
For any $\mathfrak{a}\in \mathfrak{I}_\alpha^0$ we shall consider the linear functional
\begin{equation}
    \mathcal{T}_\alpha(\mathfrak{a})\;=\;\int_{\Omega_\alpha^\circ} \dd\mu_\alpha(\omega) \mathfrak{a}_{(0,0)}(\omega)
\end{equation}
where  $\mu_\alpha$ is the interface measure on $\Omega_\alpha^\circ$ described in Proposition \ref{prop: invariant measure}, and $\mathfrak{a}_{(0,0)}$ is the image of $\mathfrak{a}$ under the conditional expectation $E\colon \mathfrak{I}_\alpha\to C_0(\Omega_\alpha^\circ)$. From the properties of $\mu_\alpha$ it follows that $\mathcal{T}_\alpha$ is a faithful normal semi-finite (f.n.s.) and lower semi-continuous $\Z^2$-invariant trace on $\I_\alpha$. We shall denote the dense domain of this trace as $L^1(\mathfrak{I}_\alpha)$, \ie the closure of $\I_\alpha^0$ under the norm
\begin{equation}
    \|\mathfrak{a}\|_1\;:=\;\mathcal{T}_\alpha(|\mathfrak{a}|)\;.
\end{equation}
 By following some arguments in the proof of \cite[Theorem 1.4]{RX} we will show the uniqueness of this trace.
\begin{proposition}\label{prop: interface trace}
    $\mathcal{T}_\alpha$ is the unique $\Z^2$-invariant trace on $\I_\alpha$, up to a positive constant.
\end{proposition}
\begin{proof}
  Let $\mathcal{T}'$ be any $\Z^2$-invariant trace. For $n\in \Z^2$ one has
\begin{equation*}
    \begin{split}
\mathcal{T}'\big(\mathfrak{r}_n\mathfrak{s}_1\big)\;&=\; \mathcal{T}'\big(\mathfrak{r}_n\mathfrak{s}_1\mathfrak{r}_n\big)\;=\; \mathcal{T}'\big(\mathfrak{s}_1\mathfrak{r}_{n-e_1}\mathfrak{r}_n\big)
    \end{split}
\end{equation*}
so an inductive argument shows that $\mathcal{T}'(\mathfrak{r}_n\mathfrak{s}_1)=0$ and in the same vein $$\mathcal{T}'\big(\mathfrak{r}_n\mathfrak{s}_2\big)\;=\;\mathcal{T}'\big(\mathfrak{l}_n\mathfrak{s}_1\big)\;=\;\mathcal{T}'\big(\mathfrak{l}_n\mathfrak{s}_2\big)\;=\;0\;, \qquad \;\forall \,n\in \Z^2$$
The above implies that $\mathcal{T}'(\mathfrak{a}\mathfrak{s}^{n})=0$ for all $n\in \Z^2\setminus\{(0,0)\}$ and $\mathfrak{a}\in C_c(\Omega_\alpha^\circ)$. In particular, for any $\mathfrak{a}=\sum_{n\in \Lambda} \mathfrak{a}_n\mathfrak{s}^{n}\in \mathfrak{I}_\alpha^0$  one has
$\mathcal{T}'(\mathfrak{a})=\mathcal{T}'(\mathfrak{a}_{(0,0)})$. Then the Riesz representation theorem implies that there exists a unique Borel regular $\Z^2$-invariant measure $\mu$ on $\Omega_\alpha^\circ$ such that
$$\mathcal{T}'(\mathfrak{a})\;=\;\int_{\Omega_\alpha^\circ}\dd \mu(\omega)\mathfrak{a}_{(0,0)}(\omega)$$
So we can conclude by Proposition \ref{prop: invariant measure} that $\mathcal{T}'$ is up to positive scaling factor,  equal to $\mathcal{T}_\alpha.$
\end{proof}
An explicit description of $\mathcal{T}_\alpha$ for $\alpha$ rational can be obtained as a direct consequence of Theorem \ref{teo: interface} and Proposition \ref{prop: interface trace}. In fact, $\mathcal{T}_\alpha$ under the isomorphism \eqref{interfae rational} agrees with  the trace $\tau_0\otimes {\rm Tr}_{\ell^2(\n{X}_\alpha)}$, where $\tau_0$ is the natural (normalized) trace on $C(\mathbb{T})$,  \ie
$$\tau_0(f)\;:=\;\int_{\mathbb{T}}{\rm d}k\,f(k)\;,\qquad f\in C(\mathbb{T})$$ 
with ${\rm d}k$  the normalized Haar measure on $\mathbb{T}$.

\subsection{K-theory of the Iwastuka algebras}\label{secc 3.1}
In this section we present  the computation of the $K$-theory of the algebras $\A_{\alpha}$ and $\I_\alpha$. The key tool is the Pimsner-Voiculescu exact sequence \cite{Pim} and the explicit description of $\Omega_\alpha$ given in Proposition \ref{teo cantor}.  We refer to the monographs \cite{Bla,Ro, Ols} for general aspect of the $K$-theory. For previous results about the $K$-theory of $C^*$-algebras associated with magnetic fields we refer to \cite{Deni1,Dani,PRO,Tom}. 

\medskip

We start with the description of the $K$-groups for $\alpha\in \mathbb{Q}\cup\{\pm\infty\}$. For that, observe that for  $n=(n_1,n_2) \in  \Z^2$  the unitary operator $\mathfrak{s}^{n}=\mathfrak{s}_1^{n_1}\mathfrak{s}_2^{n_2}$ implements an action of the subgroup $\Z_{n_1,n_2}$ by conjugation. Moreover, let $\mathfrak{u}_{n}^\sharp:=\mathfrak{1}+\mathfrak{q}_0^{\sharp}(\mathfrak{s}^n-\mathfrak{1})$  where $\sharp$ is used to label both the projections $\mathfrak{q}_{0}$ and $\mathfrak{q}_{0}^{\bot}$ defined \eqref{eq: projections}. If $\alpha=p/q$ with $p$ and $q$
relatively prime, then the invariance under the action of $\Z_{q,p}$ verifies that $\mathfrak{u}_{(q,p)}^\sharp$ is a unitary operator in $\A_{p/q}$. When $\alpha=\pm\infty$ the same analysis gives that $\mathfrak{u}_{\pm e_2}^\sharp$ is unitary in $ \A_{\pm\infty}$.\\

The following result is obtained by rescaling the lattice and applying \cite[Theorem 4.10]{Deni1}.


\begin{proposition}
    \label{teo: k-groups rational}
    For $\alpha\in \mathbb{Q}\cup\{\pm\infty\}$ the $K$-groups of the Iwatsuka algebra are given by
    $$K_0(\A_\alpha)\;=\;\Z^3\;=\;K_1(\A_\alpha)\;.
    $$
The generators for $\alpha=p/q$ with $p$ and $q$ relatively primes are
      \begin{equation}\label{p/q}
    \begin{split}
K_0(\A_\alpha)\;&=\;\Z[\mathfrak{q}_0]\oplus\Z[\mathfrak{q}^{\bot}_0]\oplus \Z[\mathfrak{c}]\;,\\
K_1(\A_\alpha)\;&=\;\Z[\mathfrak{s}_{(-p,q)}]\oplus \Z[\mathfrak{u}_{(q,p)}]\oplus \Z[\mathfrak{u}_{(q,p)}^\bot]\;,    \end{split}
    \end{equation}
and for $\alpha =\pm\infty$ are given by
 \begin{equation}
    \begin{split}
K_0(\A_\alpha)\;&=\;\Z[\mathfrak{q}_0]\oplus\Z[\mathfrak{q}^{\bot}_0]\oplus \Z[\mathfrak{c}_\alpha]\;,\\
K_1(\A_\alpha)\;&=\;\Z[\mathfrak{s}_1]\oplus \Z[\mathfrak{u}_{e_2}]\oplus \Z[\mathfrak{u}_{e_2}^\bot] \;.   \end{split}
    \end{equation}
    In the formulas above $\mathfrak{c}$ denotes a suitable projection in $\A_\alpha\otimes {\rm Mat}_N(\C)$ for some $N\in \mathbb{N}.$
\end{proposition}
\begin{proof}
    The case $\alpha=\pm \infty$ is just \cite[Theorem 4.10]{Deni1}. So it is enough  to consider only $\alpha=p/q.$ Arguing as in the proof of Theorem \ref{teo: interface} one gets
\begin{align*}   \A_\alpha\;\simeq\;\big(C(\Omega_\alpha)\rtimes_{\sigma,\theta_\alpha}\Z_{q,p}\big)\rtimes_{\sigma,\theta_\alpha}\Z_{-p,q}\;.
\end{align*}
    Since the action of $\Z_{q,p}$ and $\Z_{-p,q}$ are implemented by the unitaries $\mathfrak{s}_{(q,p)}$ and $\mathfrak{s}_{(-p,q)}$, respectively, the one can adapt step by step the arguments used in \cite[Theorem 4.10]{Deni1} arriving to \eqref{p/q}.
\end{proof}

\medskip

Now, let us come to the irrational case. First of all, recall that $\mathfrak{F}_\alpha \simeq C(\Omega_\alpha)$ where $\Omega_\alpha$ is a Cantor set as described in Proposition \ref{teo cantor}. Therefore, $\mathfrak{F}_\alpha$ is a commutative  AF algebra \cite[chapter III]{Bla} and its $K$-theory is given by $K_0(\mathfrak{F}_\alpha)=C(\Omega_\alpha, \Z)$ and $K_1(\mathfrak{F}_\alpha)=0$. Here $C(\Omega_\alpha,\Z)$ stands for the set of $\Z$-valued continuous functions on $\Omega_\alpha$. Explicitly, one can identify $K_0(\mathfrak{F}_\alpha)$ with the projectors defined in \eqref{eq: projections}
\begin{equation}\begin{split} K_0(\mathfrak{F}_\alpha)\;&=\;\bigoplus_{n\in \Z^2}\Z[\mathfrak{q}_{n}]\oplus \Z[\mathfrak{1}]\;.
\end{split}  \end{equation}

\medskip

Thanks to the Pimsner-Voiculescu exact sequence \cite{Pim} the $K$-groups of $\mathcal{Y}_1:=\mathfrak{F}_\alpha\rtimes_{\sigma_1}\Z$ are related with the $K$-groups of $\mathfrak{F}_\alpha$  through the following exact sequence
\begin{equation}\label{seq pim}
    \xymatrix{
 K_0\big(\mathfrak{F}_\alpha\big)\ar[r]^{\beta_{1*}} & K_0\big(\mathfrak{F}_\alpha\big) \ar[r]^{i_*}& K_0\big(\mathcal{Y}_1\big)\ar[d]^{\partial_0} \\
K_1\big(\mathcal{Y}_1\big)\ar[u]^{\partial_1} &K_1\big(\mathfrak{F}_\alpha\big)\ar[l]^{i_*} &K_1\big(\mathfrak{F}_\alpha\big) \ar[l]^{\beta_{1*}}\;. }
\end{equation}
Here the vertical maps $\partial_0$ and $\partial_1$ are the connecting maps of a suitable six-term exact sequence \cite[Chapter V]{Bla}, and $\beta_1:= {\bf1}-\sigma_1$.
Therefore, replacing the known $K$-groups in \eqref{seq pim} it turns out that
$$   \xymatrix{
\Z^{\oplus \Z^2}\oplus \Z\ar[r]^{} & \Z^{\oplus \Z^2}\oplus \Z \ar[r]^{}& K_0\big(\mathcal{Y}_1\big)\ar[d] \\
K_1\big(\mathcal{Y}_1\big)\ar[u] &0\ar[l] &\;\;0\;. \ar[l]
}$$
The remaining K-groups can be computed through the following proposition.
\begin{proposition}
The image and the kernel of the map $\beta_{1*}\colon K_0(\mathfrak{F}_\alpha)\rightarrow K_0(\mathfrak{F}_\alpha)$ are given by
$${\rm Im}(\beta_{1*})\;=\;\bigoplus_{n\in \Z^2}\Z\big([\mathfrak{q}_{n}]-[\mathfrak{q}_{n+e_1}]\big)\;,\qquad {\rm Ker}(\beta_{1*})\;=\;\Z[\mathfrak{1}]\;.$$
Therefore,
$$K_0(\mathcal{Y}_1)\;=\;
\bigoplus_{i\in \Z}\Z[\mathfrak{q}_{(0,i)}]\oplus\Z[\mathfrak{1}]\;,\qquad K_1(\mathcal{Y}_1)\;=\;\Z[\mathfrak{s}_1]\;.$$
\end{proposition}
\begin{proof} By using the relations between the elements of $\mathfrak{F}_\alpha$ given in \eqref{eq: projections} one gets
\begin{equation*}
\begin{split}       \beta_{1*}\big([\mathfrak{q}_n]\big)&\;=\;[\mathfrak{q}_{n}]-[\mathfrak{q}_{n+e_1}]\;,\quad \beta_{1*}([\mathfrak{1}])\;=\;[0]
\end{split}
\end{equation*}
where $e_1=(1,0)$ and  $n\in \Z^2$. Hence  ${\rm Im}(\beta_{1*})=\bigoplus_{n\in \Z}\Z\big([\mathfrak{q}_{n}]-[\mathfrak{q}_{n+e_1}]\big)$ and the kernel of $\beta_{1*}$ is $\Z[\mathfrak{1}].$ Notice that the sequence \eqref{seq pim} yields that ${\rm Ker}(\partial_1)={\rm Im}(i_*)=0$, then $$K_1(\mathcal{Y}_1)\;=\;{\rm Im}(\partial_1)\;=\;{\rm Ker}(\beta_{1*})\;\simeq\; \Z[\mathfrak{1}]\;.$$
By taking  the isometry $v_1:=\mathfrak{s}_1\otimes \mathfrak{v}$ defined in the proof of \cite[Proposition 4.9]{Deni1} one gets  $\partial_1([\mathfrak{s}_1])=-[\mathfrak{1}]$. Consequently, $K_1(\mathcal{Y}_1)=\mathbb{Z}[\mathfrak{s}_1].$ Finally one has
\begin{equation*}
\begin{split}K_0(\mathcal{Y}_1)&\;=\;K_0(\mathfrak{F}_\alpha)/{\rm Im}(\beta_{1*})\;=\;\bigoplus_{i\in \Z}\Z[\mathfrak{q}_{(0,i)}]\oplus \Z[\mathfrak{1}]
\end{split}
\end{equation*}
as claimed.
\end{proof}
Since $\A_\alpha\simeq\mathcal{Y}_1\rtimes_{\sigma_2}\Z$, the Pimsner-Voiculescu exact sequence implies
\begin{equation}\label{seq 3,8}
    \xymatrix{
 K_0\big(\mathcal{Y}_1\big)\ar[r]^{\beta_{2*}} & K_0\big(\mathcal{Y}_1\big) \ar[r]^{i_*}& K_0\big(\A_\alpha\big)\ar[d]^{\partial_0} \\
K_1\big(\A_\alpha\big)\ar[u]^{\partial_1} &K_1\big(\mathcal{Y}_1\big)\ar[l]^{i_*} &K_1\big(\mathcal{Y}_1\big) \ar[l]^{\beta_{2*}} }
\end{equation}
with $\beta_2:= {\bf1}-\sigma_2$. Replacing the known $ K$-groups it follows that
$$   \xymatrix{
\Z^{\oplus \Z}\oplus \Z \ar[r]^{} & \Z^{\oplus \Z}\oplus \Z \ar[r]^{}& K_0\big(\A_\alpha\big)\ar[d] \\
K_1\big(\A_\alpha\big)\ar[u] &\Z\ar[l] &\;\;\Z\;. \ar[l] }$$
Now we have all the elements to compute the $K$-groups of $\A_\alpha$.
\begin{proposition}
   \label{k-groups A}
The image and the kernel of the map $\beta_{2*}\colon K_0(\mathcal{Y}_1)\rightarrow K_0(\mathcal{Y}_1)$ are given by
$${\rm Im}(\beta_{2*})\;=\;\bigoplus_{i\in \Z}\Z\big([\mathfrak{q}_{(0,i)}]-[\mathfrak{q}_{(0,i+1)}]\big)\;,\qquad {\rm Ker}(\beta_{2*})\;=\;\Z[\mathfrak{1}]\;.$$
Therefore,
$$K_0(\A_\alpha)\;=\;\Z[\mathfrak{q}_0]\oplus\Z[\mathfrak{q}^{\bot}_0]\oplus \Z[\mathfrak{c}]
\;,\qquad K_1(\A_\alpha)\;=\;\Z[\mathfrak{s}_1]\oplus \Z[\mathfrak{s}_2]\;,$$
where $\mathfrak{c}$ is a projection in $\A_\alpha\otimes {\rm Mat}_N(\C)$ for some $N\in \mathbb{N}.$ 
\end{proposition}
\begin{proof} We know from \eqref{eq: projections} that
\begin{equation*}
\begin{split}
       \beta_{2*}\big([\mathfrak{q}_{(0,i)}]\big)\;=\;[\mathfrak{q}_{(0,i)}]-[\mathfrak{q}_{(0,i+1)}]\;,\quad        \beta_{2*}\big([\mathfrak{1}]\big)\;=\;[0]
\end{split}
\end{equation*}
for $i\in \Z$. It in turn that
$${\rm Im}(\beta_{2*})\;=\;\bigoplus_{i\in \Z}\Z\big([\mathfrak{q}_{(0,i)}]-[\mathfrak{q}_{(0,i+1)}]\big)\;,\qquad {\rm Ker}(\beta_{2*})\;=\;\Z[\mathfrak{1}]\;,$$
As a consequence
\begin{equation}\label{kkk}
    \begin{split}
        K_1(\A_\alpha)/{\rm Im}(i_*)\;\simeq\; {\rm Im}(\partial_1)\;\simeq\; {\rm Ker}(\beta_{2*})\;=\;\mathbb{Z}[\mathfrak{1}]\;.
    \end{split}
\end{equation}
The map $\beta_{2*}\colon K_1(\mathcal{Y}_1)\rightarrow K_1(\mathcal{Y}_1)$ meets
$$ \beta_{2*}\big([\mathfrak{s}_{1}]\big)\;=\;[\mathfrak{s}_1]-[\mathfrak{s_2\mathfrak{s}_1\mathfrak{s}_2^*}]\;=\;[\mathfrak{s}_1]-[\overline{\mathfrak{f}_B}\mathfrak{s}_1]\;=\;[\mathfrak{0}]\;,$$
where we have used that $[\overline{\mathfrak{f}_B}]=[\mathfrak{1}]$ with the homotopy $t\mapsto \expo{-\ii b_+ t}\mathfrak{q}_0+\expo{-\ii b_- t}\mathfrak{q}_0^\bot.$
Thus $\beta_{2*}\colon\Z[\mathfrak{s}_1]\rightarrow\Z[\mathfrak{s}_1]$ is the trivial homomorphism and consequently ${\rm Im}(i_*)\simeq \Z[\mathfrak{s}_1]$. Furthermore, invoking again the same argument in the proof of \cite[Proposition 4.9]{Deni1} with the isometry $v_2=\mathfrak{s}_2\otimes \mathfrak{v}$ one gets that  $\partial_1([\mathfrak{s}_2])=-[\mathfrak{1}]$ and $K_1(\A_\alpha)=\Z[\mathfrak{s}_1]\oplus \Z[\mathfrak{s}_2]$ in view of \eqref{kkk}. 
Now let us compute $K_0(\A_\alpha)$. Thanks to \eqref{seq 3,8}, one has
\begin{equation}\label{nnn}
    K_0(\A_\alpha)/{\rm Im}(i_*)\;\simeq \;{\rm Im}(\partial_0)\;\simeq \;{\rm Ker}(\beta_{2*})\;=\;\mathbb{Z}[\mathfrak{s}_1]\;,
\end{equation}
where recall that $\beta_{2*}\colon\Z[\mathfrak{s}_1]\rightarrow\Z[\mathfrak{s}_1]$ is the zero map. Since $${\rm Im}(i_*)\;=\;K_0(\mathcal{Y}_1)/{\rm Im}(\beta_{2*})\;=\;\Z[\mathfrak{q}_{0}]\oplus \Z[\mathfrak{q}^\bot_{0}]\;. 
$$
Finally, for some $N\in \N$ there is a projection $\mathfrak{c}\in\mathfrak{A}\otimes {\rm Mat}_N(\C)$ such that $\partial_0([\mathfrak{c}])=-[\mathfrak{s}_1]$ and by \eqref{nnn} it holds that  $K_0(\A_\alpha)=\Z[\mathfrak{q}_0]\oplus\Z[\mathfrak{q}^\bot_0]\oplus \Z[\mathfrak{c}]$.
\end{proof}

\subsection{K-theory of the interface algebra}
Consider the six-term exact sequence related to \eqref{seq a}
\begin{equation}\label{seq six}
    \xymatrix{
K_0(\mathfrak{I}_{\alpha})\ar[r]^{\hspace{-0,5cm}i_*} &  K_0(\A_{\alpha}) \ar[r]^{\hspace{0,1cm}{\rm ev}_*}& K_0(\A_{\rm bulk})\ar[d]^{\mathtt{exp}} \\
K_1(\mathfrak{A}_{\rm bulk})\ar[u]^{\mathtt ind} & K_1(\A_{\alpha})
\ar[l]^{\hspace{-0,1cm}{\rm ev}_*} & K_1(\mathfrak{I}_{\alpha})\ar[l]^{\hspace{0,5cm}i_*} }
\end{equation}
Propositions \ref{k-groups A} and \ref{teo: k-groups rational} provide the $K$-groups of $\A_\alpha$. Since $\A_{{\rm bulk}}$ is isomorphic to the direct sum of two noncommutative torus \cite[Example 2.10]{Deni1} then its K-theory is well-known \cite{Bla,Ols}:
$$
\begin{aligned}
K_0(\A_{{\rm bulk}})\;&=\;\Z[(\mathfrak{1,0})]\oplus\Z[(\mathfrak{0,1})]\oplus\Z[(\mathfrak{p}_{b_-},\mathfrak{0})]\oplus\Z[(\mathfrak{0},\mathfrak{p}_{b_+})]\;,\\
K_1(\A_{{\rm bulk}})\;&=\;\Z[(\mathfrak{s}_{b_{-},1},\mathfrak{1})]\oplus\Z[(\mathfrak{s}_{b_{-},2},\mathfrak{1})]\oplus\Z[(\mathfrak{1},\mathfrak{s}_{b_+,1})]\oplus\Z[(\mathfrak{1},\mathfrak{s}_{b_+,2})]\;,
\end{aligned}
$$
where $\mathfrak{p}_{b_-}$ and $\mathfrak{p}_{b_+}$ are the \emph{Powers-Rieffel projections} of $\A_{b_-}$ and $\A_{b_+}$ for $b_\pm \notin 2\pi \Z$, respectively. Otherwise, these projection agrees with the \emph{Bott projection} on $C(\T^2).$
It follows that the sequence \eqref{seq six} for $\alpha$ irrational turns out to be
$$ \xymatrix{
 K_0(\mathfrak{I}_\alpha)\ar[r]^{i_*} & \Z^3 \ar[r]^{{\rm ev}_*}& \Z^4\ar[d]^{{\mathtt{exp}}} \\
\Z^4\ar[u]^{{\mathtt{ind}}} &\Z^2
\ar[l]^{{\rm ev}_*} &K_1(\mathfrak{I}_\alpha) \ar[l]^{i_*}\;, }$$
while for $\alpha$ rational
$$ \xymatrix{
 K_0(\mathfrak{I}_\alpha)\ar[r]^{i_*} & \Z^3 \ar[r]^{{\rm ev}_*}& \Z^4\ar[d]^{{\mathtt{exp}}} \\
\Z^4\ar[u]^{{\mathtt{ind}}} &\Z^3
\ar[l]^{{\rm ev}_*} &K_1(\mathfrak{I}_\alpha) \ar[l]^{i_*}\;. }$$
Observe that in both cases the elements $(\mathfrak{0,1})$ and $(\mathfrak{1,0})$ in $\A_{\rm bulk}$ lift in the projections $\mathfrak{q}_0$ and $\mathfrak{q}^{\bot}_0$, respectively. Therefore, using  \cite[Proposition 12.2.2]{Ro} one finds 
\begin{equation}\label{expo}
{\mathtt{exp}}([(\mathfrak{1,0})])\;=\;{\mathtt{exp}}([(\mathfrak{0,1})])\;=\;[\mathfrak{1}]\;.
\end{equation}
On the other hand, we know that 
\begin{equation*}
    \begin{split}
        \mathfrak{p}_{b_+}\;&=\;\mathfrak{s}^*_{b_+,1}f(\mathfrak{s}_{b_+,2})+g(\mathfrak{s}_{b_+,2})+f(\mathfrak{s}_{b_+,2})\mathfrak{s}_{b_+,1}\;,\\
        \mathfrak{p}_{b_-}\;&=\;\mathfrak{s}^*_{b_-,1}f(\mathfrak{s}_{b_-,2})+g(\mathfrak{s}_{b_-,2})+f(\mathfrak{s}_{b_-,2})\mathfrak{s}_{b_-,1}\;,
    \end{split}
\end{equation*}
for some suitable continuous real functions $f$ and $g$ on $\mathbb{T}$ \cite[Proposition 12,4]{gracia-varilly-figueroa-01}. Consider now the self-adjoint lift of $(\mathfrak{p}_{b_-},\mathfrak{p}_{b_+})$ given by
\begin{equation*}
    \begin{split}\mathfrak{o}\;&=\;\mathfrak{s}^*_{1}f(\mathfrak{s}_{2})+g(\mathfrak{s}_{2})+f(\mathfrak{s}_{2})\mathfrak{s}_{1}\,.
    \end{split}
\end{equation*}
From the construction of $f$ and $g$ it follows that $\mathfrak{o}$ is a projection in $\A_\alpha$, and for this reason one has
\begin{equation*}
\begin{split}
    {\mathtt{exp}}\big([(\mathfrak{p}_{b_-},\mathfrak{0})]+[(\mathfrak{0},\mathfrak{p}_{b_+})]\big)\;&=\;   {\mathtt{exp}}\big([(\mathfrak{p}_{b_-},\mathfrak{p}_{b_+})]\big)\;=\;[\expo{2\pi \ii \mathfrak{o}}]\;=\;[\mathfrak{1}]\;.
\end{split}
\end{equation*}
As a result $ {\mathtt{exp}}\big([(\mathfrak{p}_{b_-},\mathfrak{0})]\big)=-{\mathtt{exp}}\big([(\mathfrak{0},\mathfrak{p}_{b_+})]\big)$ and 
\begin{equation}
  {\rm Im}(\mathtt{exp})\;=\;\Z[\mathtt{exp}(\mathfrak{p}_{b_-},\mathfrak{0})]\;=\;\Z[\mathtt{exp}(\mathfrak{0},\mathfrak{p}_{b_+})]  \;.
\end{equation}
Furthermore, Propositions \ref{teo: k-groups rational} and \ref{k-groups A} implies that ${\rm ev}_*\colon K_1(\A_\alpha)\to K_1(\A_{\rm bulk})$ is injective, so one concludes that 
\begin{equation}\label{exponential}
    K_1(\mathfrak{I}_\alpha)\;=\;\Z[\mathtt{exp}(\mathfrak{p}_{b_-},\mathfrak{0})]\;=\;\Z[\mathtt{exp}(\mathfrak{0},\mathfrak{p}_{b_+})]\;.
\end{equation}
Let us introduce a unitary element $\mathfrak{w}_\alpha$ in $\I_\alpha^\sim$ which meets the following condition
\begin{equation}\label{eq: generator}[\mathfrak{w}_\alpha]\;=\;\mathtt{exp}\big([(\mathfrak{0},\mathfrak{p}_{b_+})]\big)\;=\;\mathtt{exp}\big(-[(\mathfrak{p}_{b_-},\mathfrak{0})]\big)
\end{equation}
The unitary  $\mathfrak{w}_\alpha$ can be considered as \virg{approximate translation along the interface} since for rational values of $\alpha$ there is an explicit representative of $[\mathfrak{w}_\alpha]$ as a unitary which acts translating in the direction of the interface.

\medskip

To go further, let us now compute the $K$-groups of the interface algebra for rational slopes.  Consider 
\begin{equation}
    \mathfrak{p}_\alpha\;=\;\sigma_{(-p,q)}(\mathfrak{q}_0)-\mathfrak{q}_0\;\in\;\A_\alpha
\end{equation}
This element is a projection when $ \alpha=p/q$ and  allows to define the unitary operator $\mathfrak{w}'_\alpha$ in the unitization of $\I_\alpha$ via
\begin{equation}\label{unitario interface}\mathfrak{w}'_\alpha\;:=\;\mathfrak{1}+(\mathfrak{s}_{(q,p)}-\mathfrak{1})\mathfrak{p}_\alpha\;.
\end{equation}
It is important to point out that $\mathfrak{w}'_\alpha$ acts on $\ell^2(\Z^2)$ as a translation along the interface.
\begin{proposition}\label{prop: k interface racional}
If $\alpha$ is rational, then the classes $[\mathfrak{w}'_\alpha]=[\mathfrak{w}_\alpha]$ agrees in $K_1(\I_\alpha)$ and the $K$-theory of $\I_\alpha$ is 
    \[
     K_0(\I_\alpha)\;=\;\Z[\mathfrak{p}_\alpha]\;\simeq\;\Z\;,\qquad
         K_1(\I_\alpha)\;=\;\Z[\mathfrak{w}'_\alpha]\;\simeq\;\Z\;.
        \]
\end{proposition}
\begin{proof}
The case $\alpha=\pm \infty$ follows from Theorem \ref{teo: interface} and \cite[Proposition 4.14]{Deni1}. Now for $\alpha=p/q$ we first consider the identification $C(\T)\simeq C^*\big(\mathfrak{s}_{(q,p)}\big)$. Let us denote by $i\colon \I_\alpha\simeq C(\T)\otimes \mathfrak{K}(\ell^2(\n{X}_\alpha))$ the isomorphism provided by Theorem \ref{teo: interface}. Then, up to stabilization, one has $i_*([\mathfrak{p}_\alpha])=[\mathfrak{1}]$ and $i_*([\mathfrak{w}'_\alpha])=[\mathfrak{s}_{(q,p)}]$, which implies that
$$K_0(\I_\alpha)\;=\;\Z[\mathfrak{p}_\alpha],\qquad K_1(\I_\alpha)\;=\;\Z[\mathfrak{w}'_\alpha]\;. $$
To conclude that the classes $[\mathfrak{w}'_\alpha]=[\mathfrak{w}_\alpha]$ agrees in $K_1(\I_\alpha)$  is enough to see that both unitaries have the same noncommutative winding number by Remark \ref{remark 1} and Theorem \ref{Teo el 1}.
\end{proof}

\begin{proposition}\label{prop: k interface irrational}
   For $\alpha$ irrational it holds true that
$$K_0(\mathfrak{I}_\alpha)\;=\;\Z[\mathfrak{r}_0]\oplus\Z[\mathfrak{l}_0]\;\simeq\;\Z^2\;,\qquad K_1(\mathfrak{I}_\alpha)\;=\;\Z[\mathfrak{w}_\alpha]\;\simeq\;\Z\;.$$ 
\end{proposition}
\begin{proof} Since ${\rm ev}_*\colon K_0(\A_\alpha)\to K_0(\A_{\rm bulk})$ in the sequence \eqref{seq six} is injective by Proposition \ref{k-groups A}, then the map ${i}_*\colon K_0(\I_\alpha)\to   K_1(\A_{\alpha})$  is trivial. Therefore, $K_0(\I_\alpha)=\Z^2$ where the generators can be identified with the image under the index map of the classes $[(\mathfrak{1},\mathfrak{s}_{b_+,1})]$ and $[(\mathfrak{1},\mathfrak{s}_{b_+,2})]$.  In order to have an explicit description of the index of such elements, consider $\mathfrak{v}_i:=(\mathfrak{s}_i-\mathfrak{1})\mathfrak{q}_0+\mathfrak{1}\in \A_\alpha$. One can see that
$${\rm ev}(\mathfrak{v}_i)\;=\;(\mathfrak{1},\mathfrak{s}_{b_+,i})\;,\qquad i=1,2\;.$$
and moreover, $\mathfrak{v}_2\mathfrak{v}_2^*=\mathfrak{l}_0$ and $\mathfrak{v}^*_2\mathfrak{v}_2=\mathfrak{1}$, and  in turn
\begin{equation*}
\begin{split}
\mathtt{ind}\big([(\mathfrak{1},\mathfrak{s}_{b_+,2})])\;&=\;[\mathfrak{1}-\mathfrak{v}_2\mathfrak{v}_2^*]-[\mathfrak{1}-\mathfrak{v}_2^*\mathfrak{v}_2]\;=\;[\mathfrak{l}_0]\;.
\end{split}
\end{equation*}
On the other hand, for $\alpha>0$ the element $\mathfrak{v}_1$ is a partial isometry such that $\mathfrak{v}_1\mathfrak{v}_1^*=\mathfrak{r}_0$ and $\mathfrak{v}^*_1\mathfrak{v}_1=\mathfrak{1}$. This verifies  $\mathtt{ind}\big([(\mathfrak{1},\mathfrak{s}_{b_+,1})])=[\mathfrak{r}_0]$.
When $\alpha<0$ the element $\mathfrak{v}_1$ is no longer a partial isometry but one can take instead the partial isometry $\mathfrak{v}^\bot_i:=(\mathfrak{s}_i-\mathfrak{1})\mathfrak{q}^{\bot}_0+\mathfrak{1}$ and concludes that $\mathtt{ind}([(\mathfrak{1},\mathfrak{s}_{b_+,1}])=[\mathfrak{r}_0]$.
\end{proof}

\section{Bulk-interface correspondence}\label{sec: currents}
In this section, we derive the bulk-interface correspondence by extracting and relating the numerical invariants of the $K$-groups of the algebras $\mathfrak{I}_\alpha$ and $\A_{{\rm bulk}}.$
\subsection{Bulk topological invariants} 

Let us denote with $\mathcal{T}_\pm$  the traces per unit volume on $\A_{b_\pm}$ provided by  \cite[Proposition 2.28]{Deni1}, respectively. Given a  projection $\mathfrak{p}=(\mathfrak{p}_-,\mathfrak{p}_+)\in \mathcal{C}^1(\A_{b_-})\oplus \mathcal{C}^1(\A_{b_+})\subset \A_{\rm bulk}$, the \emph{transverse Hall conductance} associated with $\mathfrak{p}$ is defined by
\begin{equation*}
     \sigma_{{\rm bulk}}(\mathfrak{p})\;:=\;\big(\sigma_{b_{-}}(\mathfrak{p}),\sigma_{b_+}(\mathfrak{p})\big)\;
\end{equation*}
where
\begin{equation}
\begin{split}
     \sigma_{b_{\pm}}(\mathfrak{p})\;:&=\;2\pi \ii \frac{e^2}{h}\mathcal{T}_\pm\big(\mathfrak{p}_\pm[\nabla_1\mathfrak{p}_\pm,\nabla_2\mathfrak{p}_\pm]\big)\;,
\end{split}
\end{equation}
$e>0$ is the magnitude of the electron charge, and $h=2\pi\hbar$ is the Planck's constant.
\medskip

\begin{assumption}\label{BGC}
   Let $\mathfrak{h}=(\mathfrak{h}_-,\mathfrak{h}_+)$ be the \emph{bulk magnetic Hamiltonian}, \ie a selfadjoint element in $\A_{\rm bulk}$. We assume that there is a compact interval $\Delta\subset \R$ such that
$$\min {\rm Spec}(\mathfrak{h})\;< \;\min \Delta\;<\;\max \Delta\;<\;\max {\rm Spec}(\mathfrak{h})\;$$
and $\Delta\cap {\rm Spec}(\h)=\emptyset.$ Here ${\rm Spec}(\mathfrak{h})={\rm Spec}(\mathfrak{h}_-)\cup{\rm Spec}(\mathfrak{h}_+)$ stands for the spectrum of $\mathfrak{h}.$
\end{assumption}
Thanks to assumption \ref{BGC}, one has that for any $\mu\in\Delta $ the \emph{Fermi projection}
\begin{equation}
    \mathfrak{p}_\mu\;:=\;\big(\mathfrak{p}_{\mu_-},\mathfrak{p}_{\mu_+}\big)\;=\;\big(\,\chi_{(-\infty,\mu]}(\mathfrak{h}_-)\,,\,\chi_{(-\infty,\mu]}(\mathfrak{h}_+)\,\big)
\end{equation}
is an element of the bulk algebra. Moreover, suppose $\hat{\mathfrak{h}}\in \mathcal{C}^1(\A_\alpha)$ is a selfadjoint lift of $\mathfrak{h}$ under the evaluation map ${\rm ev}$. In that case, it follows that $\sigma_{b_{\pm}}(\mathfrak{p}_\mu)$ are well-defined given that the evaluation map and the functional calculus preserve the regularity. A remarkable property is that the quantities $\sigma_{b_{\pm}}(\mathfrak{p}_\mu)$ are constant on the equivalence class of $\mathfrak{p}_\mu$ in $K_0(\A_{\rm bulk})$. Indeed, one has the equality
\begin{equation}
    \sigma_{b_{\pm}}(\mathfrak{p}_\mu)\;=\;\frac{e^2}{h}{\rm Ch}(\mathfrak{p}_{\mu_\pm})
\end{equation}
 where ${\rm Ch}(\mathfrak{p}_{\mu_\pm})$ stands for the \emph{Chern number} of the projector $\mathfrak{p}_{\mu_\pm}$. Furthermore,  it is well-known that the Chern number induced a homomorphism ${\rm Ch}\colon K_0(\A_{b_\pm})\to \Z$ \cite{Con,PRO}. 
We will refer to $\sigma_{b_{\pm}}(\mathfrak{p}_\mu)$  as the \emph{bulk magnetic invariants} of the system.
\medskip

Assumption \ref{BGC} also assures the existence of a nondecreasing smooth function $g\colon\R\rightarrow [0,1]$ so that $g=0$ below $\Delta$ and $g=1$ above $\Delta$. Then
\begin{equation*}
    \begin{split}
       {\rm ev}(\mathfrak{1}-g(\mathfrak{\hat{h}}))\;&=\;\mathfrak{1}-g({\rm ev}(\mathfrak{\hat{h}}))\;=\;\mathfrak{1}-(\mathfrak{1}-\mathfrak{p}_\mu)\;=\;\mathfrak{p}_\mu\;.
    \end{split}
\end{equation*}
Observe that the Fermi projector $\mathfrak{p}_\mu$ defines a class $[\mathfrak{p}_\mu]$ in $K_0(\A_{\rm bulk})$ and the unitary operator $\mathfrak{u}_\Delta:=\expo{ 2\pi \ii g(\mathfrak{\hat{h}})}$ defines a class in $K_1(\mathfrak{I}_\alpha)$. From the fact that  $\mathfrak{1}-g(\mathfrak{\hat{h}})$ is a self-adjoint lift of $\mathfrak{p}_\mu$  one gets the following result.
\begin{proposition}\label{prop: exponential}
Let the assumption \ref{BGC} be valid and $\mu\in \Delta$. There exist a smooth function $g\colon\R\rightarrow [0,1]$ such that the unitary operator $\mathfrak{u}_\Delta=\expo{ 2\pi \ii g(\mathfrak{\hat{h}})}$ in the unitization of $ \mathfrak{I}_\alpha$ fulfills
$${\mathtt{exp}}([\mathfrak{p}_\mu])\;=\;-[\mathfrak{u}_\Delta]\;,$$
where $\mathtt{exp}\colon K_0(\A_{\rm bulk})\to K_1(\I_\alpha)$ is the exponential map related to the sequence \eqref{seq a}.
\end{proposition}

\subsection{Interface currents} From Proposition \ref{prop: interface trace} one knows that $\I_\alpha$ has (up to a normalization) a unique \emph{f.n.s.} trace $\mathcal{T}_\alpha$ which is $\Z^2$-invariant. Now the differential structure on $\I_\alpha$ is provided by the unbounded directional derivation  
$$\nabla_\alpha\;:=\;{v}^\bot_\alpha\cdot \nabla\;=\;\frac{1}{\sqrt{1+\alpha^2}}\big(-\alpha \nabla_1+\nabla_2\big)$$ 
for $\alpha\in \R$. Here ${v}^\bot_\alpha:=\frac{1}{\sqrt{1+\alpha^2}}(-\alpha,1)\in \R^2$ 
stands for the (normalized) normal vector to the interface.
For $\alpha=\pm\infty$ one has that
 $\nabla_{\pm\infty }=\pm \nabla_1$    and ${v}^\bot_{\pm\infty}=\pm(1,0)$.  We can extend such derivations on the unitization of $\mathfrak{I}_\alpha$ denoted with $\I_\alpha^\sim$. This is done  with the prescription $\nabla_{\alpha}(\mathfrak{1})=0.$  Notice that the $\Z^2$-invariance of the interface trace $\mathcal{T}_\alpha$ verifies that $\mathcal{T}_\alpha\circ \nabla_\alpha=0$ for all $\alpha\in \overline{\R}$. 

\medskip

For $k\in \mathbb{N}$, let us introduce the spaces
$$\mathcal{C}^k_\alpha\;:=\;\big\{\,\mathfrak{a}\in\mathfrak{I}_\alpha\;\big|\;\nabla_{\alpha}^k(\mathfrak{a})\in L^1(\I_\alpha)\,\big\}\;.$$
Let $\mathfrak{u}\in \mathfrak{I}^\sim_\alpha$ be a unitary element such that $\mathfrak{u}-\mathfrak{1}\in \mathcal{C}^1_\alpha$. The \emph{noncommutative winding numbers} of $\mathfrak{u}$ is defined as 
\begin{equation}
    \mathcal{W}_{\alpha}(\mathfrak{u})\;:=\;\ii\mathcal{T}_\alpha(\mathfrak{u}^*\nabla_{\alpha}(\mathfrak{u}))\;.
\end{equation}
The quantity $\W_\alpha(\mathfrak{u})$ is constant on the equivalence class of $\mathfrak{u}$ in $K_1(\I_\alpha).$  Namely, it is well-known that \cite{PRO}
\begin{equation}
\W_\alpha(\mathfrak{u})\;=\;\zeta_\alpha(\mathfrak{u}^*-\mathfrak{1},\mathfrak{u}-\mathfrak{1})\;=\;\langle [\zeta_\alpha],[\mathfrak{u}]\rangle
\end{equation}
where the prescription $\zeta_\alpha(\mathfrak{a},\mathfrak{b})=\ii\mathcal{T}_\alpha(\mathfrak{a}\nabla_\alpha \mathfrak{b})$ defines a cyclic $1$-cocycle $[\zeta_\alpha]$ in the cyclic cohomology of a dense subalgebra of $\I_\alpha$\footnote{See \cite[Chapter 4]{Tom} for a  discussion of the domain of these cocycles.}, and the pairing of $[\zeta_\alpha]$ and $[\mathfrak{u}]$ verifies that $\W_\alpha(\mathfrak{u})$ is a numerical invariant that depends on the class $[\mathfrak{u}]\in K_1(\I_\alpha)$ .

\begin{remark}\label{remark 1}
  Proceeding as in \cite{Deni1, Guo,Dani}, one can verify that $\W_\alpha(\mathfrak{w}'_\alpha)=1$ for any $\alpha\in \mathbb{Q}$, where $\mathfrak{w}'_\alpha$ is the unitary operator defined in \ref{unitario interface}. Indeed, from Proposition \ref{prop: k interface racional} when $\alpha=p/q$  one has that
\begin{align*}
\nabla_\alpha\mathfrak{w}'_\alpha\;&=\;\frac{1}{c_\alpha}\mathfrak{s}_{(q,p)}\mathfrak{p}_\alpha\;
\end{align*}
where the normalization constant is defined in \eqref{eq:def_c}.
Then 
\begin{align*}
    \W_\alpha(\mathfrak{w}'_\alpha)\;&=\; \mathcal{T}_\alpha\big((\mathfrak{w}'_\alpha)^*\nabla_\alpha\mathfrak{w}'_\alpha\big)\;=\;\frac{1}{c_\alpha}\mathcal{T}_\alpha\big((\mathfrak{p}_\alpha\mathfrak{s}_{(q,p)}^*+\mathfrak{p}_\alpha^\bot)\mathfrak{s}_{(q,p)}\mathfrak{p}_\alpha\big)\\
&=\;\frac{1}{c_\alpha}\mathcal{T}_\alpha\big(\mathfrak{p}_\alpha\big)\;=\;1
\end{align*}
where the last equality follows from the fact that the image of $\mathfrak{p}_\alpha$ in $C(\Omega^\circ_\alpha)$ is given by the characteristic function of a single point. The case $\alpha=\pm \infty$  is the same as \cite[Lemma 4.18]{Deni1}.  \hfill $\blacktriangleleft$
\end{remark}

The (density of) \emph{interface current}   along the interface is defined by
\begin{equation}
    J_{\alpha}(\Delta)\;:=\;\frac{e}{\hslash}\mathcal{T}_\alpha\big( g'(\mathfrak{\hat{h}})\nabla_{\alpha}(\hat{\mathfrak{h}})\big)\;,
\end{equation}
where $g$ is any smooth function that meets the conditions of Proposition \ref{prop: exponential}. Therefore, Kubo's Formula \cite{Bel,Deni2} implies that the terms $\sigma:=eJ_{\alpha}$ provide the  \emph{interface conductances}.  If in addition one consider Assumption \ref{BGC}, one obtains
$$\mathcal{T}_\alpha\big(g'(\mathfrak{\hat{h}})\nabla_{\alpha}(\mathfrak{\hat{h}})\big)\;=\;-\frac{1}{2\pi}\W_\alpha(\mathfrak{u}_\Delta)\;,\qquad i=1,2\;.$$
The latter equality can be obtained by adapting the proof of \cite[Proposition 7.1.2]{PRO}. As a consequence, the interface conductance associated with the states in $\Delta$ is given by
$$\sigma(\mathfrak{u}_\Delta)\;=\;\frac{e^2}{h}\W_\alpha(\mathfrak{u}_\Delta)\;.$$
Notice that in agreement with the above, one can define for general unitary operators $\mathfrak{u}\in \mathfrak{I}^\sim_\alpha$ such that $\mathfrak{u}-\mathfrak{1}\in \mathcal{C}_{\alpha}^1$ the interface conductance by
\begin{equation}
    \sigma(\mathfrak{u})\;:=\;\frac{e^2}{h}  \W_\alpha(\mathfrak{u})\;.
\end{equation}
\begin{remark}
 The term $\sigma(\mathfrak{u})$ is the proportionality coefficient of the current flowing along the interface in the configuration $\mathfrak{u}$ \cite{Deni1, PRO}.  \hfill $\blacktriangleleft$ 
\end{remark}

\subsection{Bulk-interface correspondence}
Now we are in a position to state and prove the so-called bulk-interface correspondence for Iwatsuka magnetic fields. It will turn out that the result is independent of the slope $\alpha.$ Here an Iwatsuka Hamiltonian is  a selfadjoint element $\hat{\mathfrak{h}}$ in $\A_\alpha.$
\begin{theorem}\label{Teo current}
Let $\alpha\in \overline{\R}$ and $\hat{\mathfrak{h}}$ an Iwatsuka Hamiltonian. Assume Assumption \ref{BGC} for  the bulk Hamiltonian $\mathfrak{h}:={\rm ev}(\hat{\mathfrak{h}})$ and let $\mu\in \Delta$. If $\hat{\mathfrak{h}}$ lies in $\mathcal{C}^k(\A_{\alpha})$ for some $k\geq 1$, then the interface conductance associated with the unitary operator $\mathfrak{u}_\Delta$ defined in the Proposition \ref{prop: exponential} can be expressed as the difference of the bulk magnetic invariants of the system, \ie 
\begin{equation}
    \sigma(\mathfrak{u}_\Delta)\;=\;\sigma_{b_+}(\mathfrak{p}_\mu)-\sigma_{b_-}(\mathfrak{p}_\mu)\;.
\end{equation}
\end{theorem}
\begin{proof}
Proposition \ref{prop: exponential} yields that $\W_\alpha(\mathfrak{u}_\Delta)=-\W_\alpha(\mathtt{exp}([\mathfrak{p}_\mu])).$ Furthermore,   the exponential map is surjective \eqref{exponential} and meets  $${\mathtt{exp}}(-[(\mathfrak{p}_{b_-},\mathfrak{0})])\;=\;[\mathfrak{w}_\alpha]\;=\;{\mathtt{exp}}([(\mathfrak{0},\mathfrak{p}_{b_+})])$$ 
where the unitary $\mathfrak{w}_\alpha\in \I^\sim_\alpha$ is defined in \eqref{eq: generator}. Thus arguing as in the proof \cite[Lemma 4.5]{Dani} one obtains
$$\mathtt{exp}(-[\mathfrak{p}_\mu])\;=\;\big({\rm Ch}(\mathfrak{p}_{\mu_+})-{\rm Ch}(\mathfrak{p}_{\mu_-})\big)[\mathfrak{w}_\alpha]$$
Finally, Theorem \ref{Teo el 1} implies that  $\W_\alpha([\mathfrak{w}_\alpha])=1$ and in turn\begin{equation*}
    \begin{split}
         \sigma(\mathfrak{u}_\Delta)\;&=\;\frac{e^2}{h}   
 \mathtt{}
 \W_\alpha(\mathfrak{u}_\Delta)\;=\;\frac{e^2}{h}  \W_\alpha\big({\mathtt{exp}}(-[\mathfrak{p}_\mu])\big)\; =\;\frac{e^2}{h}\big({\rm Ch}(\mathfrak{p}_{\mu_+})-{\rm Ch}(\mathfrak{p}_{\mu_-})\big)\\
         &=\;\sigma_{b_+}(\mathfrak{p}_\mu)-\sigma_{b_-}(\mathfrak{p}_\mu)
    \end{split}
\end{equation*}
as claimed.
\end{proof}

The next result states that the bulk-interface correspondence is stable under magnetic perturbations localized near the interface. For a given $b\in C_0(\Omega^\circ_\alpha)$\footnote{It is a  function $b\colon \Z^2\to \R$ that admits a continuous extension to $\Omega^\circ_\alpha$.} let us use the short notation $\A_{\alpha,b}:=\A_{B_\alpha+b}$ for the magnetic algebra generated by the perturbed magnetic field $B_\alpha+b$. This notation is compatible with 
$\A_{\alpha,0}=\A_{\alpha}$.

\begin{corollary}\label{coro: perturbation}
For any $b\in C_0(\Omega^\circ_\alpha)$ it holds true that  $K_*\big(\A_{\alpha,b})\simeq K_*(\A_{\alpha})$. Consequently, the bulk-interface correspondence is stable under magnetic perturbation that vanishes far from the interface.
\end{corollary}
\begin{proof}
First of all, let us identify the perturbed flux operator $\mathfrak{f}_{\alpha,b}$ with the function $f_{\alpha,b}(n)=\expo{\ii B_\alpha(n)}\expo{\ii b(n)}$ on $\Z^2$. Since the function $\Z^2\ni n\mapsto\expo{\ii b(n)}$ lies in $C(\Omega_\alpha)$, then $\mathfrak{F}_{\alpha,b}\subset \mathfrak{F}_{\alpha}$. Furthermore, notice that $f_{\alpha,b}$ has the same asymptotic behavior of $f_\alpha$, so $\mathfrak{F}_{\alpha,b}$ separates the points of $\Omega_\alpha$ and as a consequence of Weierstrass Theorem $\mathfrak{F}_{\alpha,b}\simeq \mathfrak{F}_{\alpha}.$ Finally, Proposition \ref{prop: cros} and the Pimsner-Voiculescu exact sequence \cite{Pim} show that $K_*\big(\A_{\alpha,b})\simeq K_*(\A_{\alpha})$. The remaining part follows from the fact that $K_*\big(\I_{\alpha,b})\simeq K_*(\I_{\alpha})$ is also true and the interface current depends on these $K$-groups.
\end{proof}
We finish this section with an explicit model described in \cite[Example 4.9]{Dani} for which the Theorem \ref{Teo current} assures the existence of non-trivial interface currents. 
\begin{example}
In this model, we consider the values
 $b_+= 2\lambda\pi$ and $b_-=2\xi \pi$ so that $\lambda$ and $\xi$ are rational numbers and $\lambda-\xi\notin \Z$. The dynamics of this system is described by the magnetic Hamiltonian 
    \begin{equation*}
        \mathfrak{\hat{h}}\;:=\;\mathfrak{s}_1+\mathfrak{s}_1^*+\mathfrak{s}_2+\mathfrak{s}_2^*+\mathfrak{v}
    \end{equation*}
    where $\mathfrak{v}$ is a selfadjoint element in the interface algebra $\mathfrak{I}_\alpha$ such that $\mathfrak{v}$ is in the domain of $\nabla_\alpha$. Observe that the components
    of the bulk Hamiltonian $\mathfrak{h}=(\mathfrak{h}_-,\mathfrak{h}_+)\equiv {\rm ev}(\hat{\mathfrak{h}})$ fulfills
    \begin{equation*}
        \mathfrak{h}_+\;=\;\mathfrak{s}_{b_+,1}+\mathfrak{s}_{b_+,1}^*+\mathfrak{s}_{b_+,2}+\mathfrak{s}_{b_+,2}^*
    \end{equation*}
      \begin{equation*}
        \mathfrak{h}_-\;=\;\mathfrak{s}_{b_-,1}+\mathfrak{s}_{b_-,1}^*+\mathfrak{s}_{b_-,2}+\mathfrak{s}_{b_-,2}^*
    \end{equation*}
According to \cite[Section 2.1]{Deni1}, in the Landau gauge $\h_+$ reads
 \begin{equation*}
     \begin{split}
         (\h_+ \psi)(n_1,n_2)\;=&\;\psi(n_1-1,n_2)+\psi(n_1+1,n_2)+\expo{2\pi \lambda n_1}\psi(n_1,n_2-1)\\
         &\;+\expo{-2\pi \lambda n_1}\psi(n_1,n_2+1)
     \end{split}
 \end{equation*}
 for all $\psi\in \ell^2(\Z^2)$. This is a Harper-like operator \cite{Har} and the spectrum of $\h_+$ is given by the union of $q$ energy bands when $\lambda=p/q$, where $p$ and $q$ relative prime integers \cite[Section 2.6]{Bel}. Moreover, all the energy bands are separated except the central one \cite{Avr,Cho}.  
 Since the same arguments also work for $\h_-$, then there exist values of $\lambda$ and $\xi$ for which $\mu$ lies in a common spectral gap $\Delta$ of $\h_+$ and $\h_-$.  Therefore $\hat{\h}$ meets \eqref{BGC} and  ${\rm Ch}(\mathfrak{p}_{\mu_+})\neq {\rm Ch}(\mathfrak{p}_{\mu_-})$\footnote{The Hofstadter butterflies \cite{AVR1, Avr} provides specific values for $\lambda$ and $\xi$.}. Finally, Theorem \ref{Teo current} shows that  $\sigma(\mathfrak{u}_\Delta)\neq 0.$ 
\end{example}

\appendix
\section{}\label{app smooth}
In this appendix, we compute the noncommutative winding number of the unitary $\mathfrak{w}_\alpha$ defined in \eqref{eq: generator}. The strategy is to establish a connection between the $K$-theory and cyclic cohomology of $\I_\alpha$ with a suitable smooth edge algebra studied in \cite{Tom}.
\subsection{Smooth interface algebra}
For $b\in [0,2\pi)$ let $\mathcal{A}_b$ be the \emph{non-commutative torus}, \ie,  the universal $C^*$-algebra generated by two abstract unitary operators $u_1$ and $u_2$ so that 
\begin{equation*}
    u_1u_2\;=\;\expo{ \ii b}u_2u_1 
\end{equation*}
Say in other words, the algebra $\mathcal{A}_b$ is the $C^*$-algebraic completion of the algebra
spanned by the Fourier series in the noncommutative polynomials 
\begin{equation*}
    a\;=\;\sum_{n\in \Z^2}a_n u^n\;,\quad\qquad u^n=u_1^{n_1}u_2^{n_2},\;n=(n_1,n_2)\;.
\end{equation*}
It is important to point out that $\A_b$, the magnetic algebra associated with a constant magnetic field the strength $b$, agrees with a faithful representation of $\mathcal{A}_b$ on $\ell^2(\Z^2)$ \cite[Example 2.10]{Deni1}.

\medskip

Consider the unitary vectors  $v_\alpha:=\frac{1}{\sqrt{1+\alpha^2}}(1,\alpha)$  and $v_\alpha^\bot:=\frac{1}{\sqrt{1+\alpha^2}}(-\alpha,1)$ on $\R^2$, in the parallel and orthogonal direction of the interface, respectively. These two vectors generate two strongly continuous actions $\R\ni t\mapsto \tau^\|_t,\tau^\bot_t$  on $\mathcal{A}_b$ which acts on the generators via
\begin{equation}\label{eq: actio}
    \tau^\|_t(u^n)\;=\;\expo{\ii t v_\alpha\cdot n}u^n\;,\qquad\tau^\bot_t(u^n)\;=\;\expo{\ii t v_\alpha^\bot\cdot n}u^n
\end{equation}
 for $n=(n_1,n_2)\in \Z^2.$ Consequently, there is also a strongly continuous $\R^2$-action on $\mathcal{A}_b$ provided by $\tau:=\tau^\|\times\tau^\bot.$ Now the \emph{directional derivations} $\nabla_\alpha$ and $\nabla_\alpha^\bot$ on $\mathcal{A}_b$ are given by
$$\nabla_\alpha(a)\;:=\;\lim_{\epsilon \to 0}\frac{\tau_{\epsilon t}^\|(a)-a}{\epsilon}\;,\qquad \nabla_\alpha^\bot(a)\;:=\;\lim_{\epsilon \to 0} \frac{\tau_{\epsilon t}^\bot(a)-a}{\epsilon}$$
For any suitable differentiable projection $p\in \mathcal{A}_b$ the Chern cocycle associated with the action $\tau$ is given by
\begin{equation}\label{eq: chern smooth}
    {\rm Ch}_{\mathcal{T},\tau}(p)\:=\;\mathcal{T}\big(p[\nabla_\alpha p,\nabla_\alpha^\bot p]\big)
\end{equation}
where $\mathcal{T}$ is the $\tau$-invariant faithful normal semi-finite trace on $\mathcal{A}_b\simeq \A_b$ according to \cite[Proposition 2.28]{Deni1}. Given the properties of the trace and the directional derivatives, one has that ${\rm Ch}_{\mathcal{T},\tau}(\cdot)$ defines a cyclic  $2$-cocycle in a dense subalgebra $\mathcal{A}_b$, depending of the class $[p]\in K_0(\mathcal{A}_b).$ 

\medskip

Let us introduce the \emph{smooth interface algebra} defined by $\tilde{\I}_{\alpha}^b:=\mathcal{A}_b \rtimes_{\tau^\bot} \R$. Observe that this algebra agrees with the edge algebra studied in \cite[Chapter 5.2]{Tom} when the disordered space is trivial. Based on \cite{Rae}, we pursue a description of $\tilde{\I}_\alpha^b$ in terms of its representation on $\ell^2(\Z^2)$. For this purpose, let $\rho$ be the faithful representation of $\mathcal{A}_b$ on $\ell^2(\Z^2)$ so that $\rho(\mathcal{A}_b)=\A_b.$ Arguing as in \cite{Les}, let $\mathfrak{n}:=(\mathfrak{n}_1,\mathfrak{n}_2)$ be the position operator on $\ell^2(\Z^2)$ and  $D=v_\alpha^\bot\cdot \mathfrak{n}$  the selfadjoint operator which implement the $\R$-action on $\mathcal{A}_b$, \ie $$\rho(\tau_t^\bot(a))=\expo{2\pi \ii t D}\rho(a)\expo{-2\pi \ii t D}\;,\qquad a\in \mathcal{A}_b$$
  Then, following verbatim  the ideas in \cite[Section 1.1]{Tom} one can see that
    \begin{equation}\label{eq: smooth rational}
    \begin{split}
\tilde{\mathfrak{I}}_\alpha^b&\;\simeq \;C^*\big\{\, \rho(a) f(D)\;|\;a\in \mathcal{A}_b,\,f\in C_0\big({\rm Spec}(D)\big)\,\big\}\;.
   \end{split}
    \end{equation}
Observe that from its very definition, the spectrum of $D$ agrees with the closure in $\R$ of the additive subgroup 
\begin{equation*}
\Gamma_\alpha\;:=\;v_\alpha^\bot\cdot \Z^2\;=\;\big\{ (1+\alpha^2)^{-1/2}(-\alpha n_1+n_2)\,|\,n_1,n_2\in \Z\big\}\;.
\end{equation*}
Then for $\alpha\in \mathbb{Q}$  the spectrum of $D$ becomes $\Gamma_\alpha=(1+\alpha^2)^{-1/2}\mathbb{X}_\alpha$ and by \eqref{eq: smooth rational} and Theorem \ref{teo: interface} one has
$$\tilde{\mathfrak{I}}_\alpha^b\;\simeq\;C_0(\Gamma_\alpha)\rtimes_{\sigma,\theta_b}\Z^2\;\simeq \;C(\T_\alpha)\otimes \mathfrak{K}(\ell^2(\Gamma_\alpha))\;\simeq\;\I_\alpha$$
where $\T_\alpha\simeq \T$ is the Pontryagin dual group of $\Gamma_\alpha$, the action $\sigma$ is implemented by conjugation with the unitaries $\rho(u_j)$ and $\theta_b$ is the $2$-cocycle \eqref{cocycle} for the constant magnetic field $B(n)\equiv b$.  
This isomorphism can be also derived by using Takai duality \cite{Take}. Indeed, when $\alpha$ is rational the action $\tau^\bot$ becomes $\Gamma_\alpha$-periodic, yielding
\begin{equation}\label{isomorphism rational}
    \tilde{\I}_{\alpha}^b\;\simeq\;\mathcal{A}_b \rtimes_{\tau^\bot} \T_\alpha\;\simeq \;C(\T_\alpha)\otimes \mathfrak{K}(\ell^2(\Gamma_\alpha))\;\simeq\; \I_\alpha\;.
\end{equation}
In the second isomorphism, we have used Takai duality identifying $\mathcal{A}_b\simeq C(\T_\alpha)\rtimes\Gamma_\alpha$. We shall denote by $\imath' \colon  \tilde{\I}_{\alpha}^b\simeq\I_\alpha$ the isomorphism given in \eqref{isomorphism rational}.

\medskip

When $\alpha$ is irrational, the situation is different since the algebras $\tilde{\I}_\alpha^b$ and $\I_\alpha$ are no longer isomorphic. In fact they  have different  $K$-theory as proved in Proposition \ref{prop: k interface irrational} with $K_1(\I_\alpha)=\Z$ and
 \begin{equation}
   \begin{split}
K_i\big(\tilde{\I}_\alpha^b\big)\;&=\;K_i\big(\mathcal{A}_b \rtimes_{\tau^\bot}\R\big)\;\simeq\;K_{i+1}(\mathcal{A}_b\big)\;\simeq\; \Z^2\qquad i=0,1\;.
   \end{split}
   \end{equation}
Here the second isomorphism is provided by the Connes’ Thom map \cite{Con2}. The reason these algebras are non-isomorphic is that for $\alpha$ irrational, the spectrum of $D$ is equal to $\R$, leading to
  \begin{equation}\label{eq: smooth irrational}
    \begin{split}
\tilde{\mathfrak{I}}_\alpha^b&\;\simeq\;C_0(\R)\rtimes_{\sigma,\theta_b}\Z^2\;\not\simeq\;\I_\alpha
   \end{split}
    \end{equation}
where $(\sigma_nf)(x)=f(x-{v}_\alpha^\bot\cdot n)$ for $f\in C_0(\R)$. However, this representation of $\tilde{\I}_\alpha^b$ allows us to recover the $K$-theory of $\I_\alpha$  as we will state in the following result.
\begin{proposition}\label{prop: appendix}
    For $\alpha$ irrational there is a $*$-homomorphism $\imath_*\colon K_*(\tilde{\I}_\alpha^b)\to K_*(\I_\alpha) $   so that the map $\imath_*\colon K_0(\tilde{\mathfrak{I}}^b_\alpha)\to 
 K_0(\mathfrak{I}_\alpha)$ is an isomorphism and
  $\imath_*\colon K_1(\tilde{\mathfrak{I}}_\alpha^b)\to K_1(\mathfrak{I}_\alpha)$
  is  surjective with kernel  $\Z[M_f]$, where $M_f$ is the multiplication operator by the function $f(x)=\frac{x-\ii}{x+\ii}$ for $x\in \R$.
\end{proposition}
\begin{proof}
The inclusion $C_0(\R)\hookrightarrow C_0(\Omega_\alpha^\circ)$ discussed in \eqref{eq: embedding} yields a injective $*$-homomorphism $\imath'\colon \tilde{\mathfrak{I}}_\alpha^b\hookrightarrow \I_\alpha^b:=C(\Omega^\circ_\alpha)\rtimes_{\sigma,\theta_b}\Z^2$. Let $b'\in C_0(\Omega^\circ_\alpha)$ be a suitable magnetic perturbation such that $B_\alpha +b'$ is the constant $b$ on the support of the function $l_0$ whence
$$(\mathfrak{l}_0\psi)(n)\;=\;l_0(n)\psi(n),\qquad \psi\in \ell^2(\Z^2).$$
This in turn that $\mathfrak{l}_0\I_{\alpha,b'}\mathfrak{l}_0\simeq \mathfrak{l}_0\I_{\alpha}^b\mathfrak{l}_0$ and the same argument used in Theorem \ref{teo: interface} shows that $\I_{\alpha,b'}\simeq \I_\alpha^b.$ Furthermore, Corollary \ref{coro: perturbation} implies that
$$K_*(\I_\alpha)\;\simeq\;K_*(\I_{\alpha,b'})\;\simeq\;K_*(\I_\alpha^b)$$
For  $b=0$ and $\theta_b=1$, by \cite[Theorem 2.6]{RX} the  induced homomorphism
  $$\imath'_*\colon K_0(\tilde{\mathfrak{I}}^0_\alpha)\to K_0(\mathfrak{I}_\alpha^0)$$
is injective  and 
  $$\imath'_*\colon K_1(\tilde{\mathfrak{I}}_\alpha^0)\to K_1(\mathfrak{I}_\alpha^0)$$
  is  surjective with kernel  $\Z[M_f]$. To conclude is enough to consider $\imath_*$ as the composition of the following arrows
  
  \begin{equation}
    \xymatrix{
 K_*(\tilde{\I}_\alpha^0)\ar[r]^{\imath'_*} \ar[rd]^{\imath_*}&K_*(\mathfrak{I}_{\alpha}^0)\ar[d]^{\simeq}\\ & K_*(\mathfrak{I}_{\alpha})\\}
\end{equation}
For the case $b\neq 0$, the Pimsner-Voiculescu exact sequence verifies that the algebras $\tilde{ \I}_\alpha^b$ and $\I_\alpha^b$ have $K$-groups independent of $b$, so this similar to $b=0.$
\end{proof}

Let us introduce now the $C^*$-algebra $$\tilde{\A}_{\alpha,+}:=C^*\big\{\, \rho(a) f(D)\;|\;a\in \mathcal{A}_b,\,f\in C_{0,+}(\R)\,\big\}$$
where $C_{0,+}(\R)$ stands for the algebra of continuous functions which vanish in $-\infty$ and admit a limit in $+\infty$. Then it is well known that there is an exact sequence \cite[Proposition 4.13]{Tom}
\begin{equation}\label{seq smooth}
    \xymatrix{
 0\ar[r]&\tilde{\mathfrak{I}}_{\alpha}^b\ar[r]^{i} & \tilde{\A}_{\alpha,+}\ar[r]^{{\rm ev}_+}& \A_b\ar[r]&0\\}\;.
\end{equation}
It should be noticed that this sequence is the core behind the \emph{smooth bulk-boundary correspondence} proved in \cite{Tom}.

\subsection{Duality between cocycles} The twisted crossed-product structure of $\tilde{\I}_\alpha^b$ discussed in \eqref{eq: smooth rational} and \eqref{eq: smooth irrational} shows that any element $\mathfrak{a}\in \tilde{\I}_\alpha^b$ as a Fourier series of the form
\begin{equation*}
    \mathfrak{a}\;=\;\sum_{n\in \Z^2}\mathfrak{a}_n\mathfrak{u}^n\;,\qquad \mathfrak{a}_n\in C_0\big({\rm Spec}(D)\big)
\end{equation*}
where $\mathfrak{u}^n=\rho(u_1^{n_1}u_2^{n_2}).$ Let $\tilde{\mathcal{T}}_\alpha$ be the dual trace of $\mathcal{T}$ on $\tilde{\I}_\alpha^b$ according to \cite[Definition 1.16]{Take2}. We shall denote its domain as  $ L^1(\tilde{\I}^b_\alpha)$. In light of \cite[Proposition 1.5.4]{Tom} it holds that for  any $\mathfrak{a}\in L^1(\tilde{\I}^b_\alpha)$
\begin{equation}
\tilde{\mathcal{T}}_\alpha(\mathfrak{a})\;=\;c_\alpha\sum_{ \Gamma_\alpha}\mathfrak{a}_{(0,0)}(x)\;, \qquad \tilde{\mathcal{T}}_\alpha(\mathfrak{a})\;=\;\int_{\R}\dd\mu(x)\mathfrak{a}_{(0,0)}(x)
\end{equation}
for $\alpha$ rational and irrational, respectively. Here $c_\alpha$ is the smallest positive constant such that $\Gamma_\alpha=c_\alpha\Z$\footnote{For $\alpha=p/q$ it follows that $c_\alpha$ is nothing but $(p^2+q^2)^{-1/2}.$} and $\mu$ is the Lebesgue measure on $\R$. Then, from the construction of the interface trace $\mathcal{T}_\alpha$ described in Section \ref{section interface and bulk} one can check that 
\begin{align}\label{eq: traces} 
\tilde{\mathcal{T}}_\alpha\big(\imath'(\mathfrak{a})\big)\;=\;\mathcal{T}_\alpha(\mathfrak{a})\;,\qquad \mathfrak{a}\in L^1\big(\tilde{\I}_\alpha^b\big)
\end{align}
where $\imath'\colon \tilde{\I}_\alpha^b\to \I_\alpha^b$ is either the isomorphism $*$-homomorphism in \eqref{isomorphism rational} or the homomorphism presented in the proof of the Proposition \eqref{prop: appendix}. 

\medskip

In view that $\tilde{\mathcal{T}}_\alpha$ is $\tau^\|$-invariant lower semi-continuous trace, then we can define a cyclic $1$-cocycle via
\begin{equation*}
\tilde{\zeta}_\alpha(\mathfrak{a},\mathfrak{b})\;:=\;\tilde{\mathcal{T}}_\alpha\big(\mathfrak{a}^*\nabla_\alpha \mathfrak{b}\big)
\end{equation*}
which defines a class $[\tilde{\zeta}_\alpha]$ in the odd cyclic cohomology of a dense subalgebra of $\tilde{\I}_\alpha^b$. For a suitable differentiable unitary $\mathfrak{u}$ in the unitization of $\tilde{\I}_\alpha^b$ we define its \emph{non-commutative winding number} via
\begin{equation}\label{eq: winding smooth}
  \tilde{\W}_{\tilde{\mathcal{T}}_\alpha,\tau^\|}(\mathfrak{u})\;:=\;\ii \tilde{\mathcal{T}}_\alpha\big(\mathfrak{u}^*\nabla_\alpha(\mathfrak{u})\big)
\end{equation}
It turns out that this value only depends on the class $[\mathfrak{u}]\in K_1(\tilde{\I}_\alpha^b)$. Indeed, the following pairing is well-known \cite{PRO,Tom}
\begin{equation}
    \tilde{\W}_{\tilde{\mathcal{T}}_\alpha,\tau^\|}(\mathfrak{u})\;=\;\big\langle [\tilde{\zeta}_\alpha], [\mathfrak{u}]\big\rangle\;=\;\tilde{\zeta}_\alpha(\mathfrak{u}^*-{\bf 1},\mathfrak{u}-{\bf 1})
\end{equation}
  As an immediate consequence of \cite[Theorem 4.53]{Tom}, one gets the duality between Chern cocycles defined in \eqref{eq: chern smooth} and \eqref{eq: winding smooth}.

\begin{theorem}[\cite{Tom}]\label{Teo: el 2}
    Let $\mathtt{exp}\colon K_0(\A_b)\to K_1(\tilde{\I}_\alpha^b)$ be the exponential map related to the sequence \eqref{seq smooth}. Then for any projection class $[\mathfrak{p}]\in  K_0(\A_b)$
    \begin{equation}
        {\rm Ch}_{\mathcal{T},\tau}([\mathfrak{p}])\;=\;\tilde{\W}_{\tilde{\mathcal{T}}_\alpha,\tau^\|}(\mathtt{exp}([\mathfrak{p}]))
    \end{equation}
    
\end{theorem}

   
Now we are in a position to state the main result of this appendix.
\begin{theorem}\label{Teo el 1}
    For any $\alpha\in \overline{\R}$ it holds that $\W_\alpha([\mathfrak{w}_\alpha])=1$, where $\mathfrak{w}_\alpha$ is the unitary defined in \eqref{eq: generator}.
\end{theorem}
\begin{proof}
Let us start with $\alpha$ irrational.  We first notice that equation \eqref{eq: traces} and Theorem \ref{Teo: el 2} lead to
\begin{align*}
       \W_\alpha ( \imath_*\circ \mathtt{exp}([\mathfrak{p}]))\;=\;\tilde{\W}_{\mathcal{T}_\alpha,\tau^\|} (\mathtt{exp}([\mathfrak{p}]))\;=\;{\rm Ch}_{\mathcal{T},\tau}([\mathfrak{p}])
\end{align*}
where $\imath_*\colon$ is provided by Proposition \ref{prop: appendix}. Let $b'\in C_0(\Omega^\circ_\alpha)$ such that $B_\alpha +b'$ is the constant $b$ in the support of the function $l_0$. It follows that $*$-homomorphism $\imath'\colon \tilde{\I}_\alpha^b\to \I_\alpha^b\simeq \I_{\alpha,b'}$ turns out to be a non-degenerate $*$-representation of $\tilde{\I}_\alpha^b$ on $\ell^2(\Z^2)$. So by \cite[Lemma I.9.14]{Dav}, there is a unique $*$-homomorphism $\hat{\imath}'\colon \tilde{\A}_{\alpha,+}\to \A_{\alpha,b'}\subset \mathcal{B}(\ell^2(\Z^2))$ which extends $\imath'.$ Therefore,  taking $b=b_+$ in the sequence \eqref{seq smooth} one arrives in the following commutative diagram
\begin{equation}
        \xymatrix{
 0\ar[r]&\tilde{\mathfrak{I}}_{\alpha}^{b_+}\ar[r]^{i}\ar[d]_{\imath'} & \tilde{\A}_{\alpha,+}\ar[r]^{{\rm ev}_+}\ar[d]_{\hat{\imath}'}& \A_{b_+}\ar[r]\ar[d]_{i_+}&0 \\
  0\ar[r]&\mathfrak{I}_{\alpha,b'}\ar[r]^{i} & \A_{{\alpha,b'}}\ar[r]^{{\rm ev}}& \A_{\rm bulk}\ar[r]&0 
 }
\end{equation}
 where $i_+(\mathfrak{a})=(\mathfrak{0},\mathfrak{a})$. The naturalness of the exponential map with Corollary \ref{coro: perturbation} yield
\begin{equation*}
       \xymatrix{
 & K_0( \A_{b_+})\ar[d]_{(i_+)_*}\ar[r]^{\mathtt{exp}} & K_1\big(\tilde{\I}_\alpha^{b_+}\big)\ar[d]^{\imath_*} \\
& K_0(\A_{\rm bulk})\ar[r]^{\mathtt{exp}}\ar[d]^\simeq &  K_1(\I_{\alpha,b'})\ar[d]^\simeq
\\
& K_0(\A_{\rm bulk})\ar[r]^{\mathtt{exp}} &  K_1(\I_\alpha)}
\end{equation*}
Thus the above verifies that
\begin{equation*}
\begin{split}
    \mathtt{}\W_\alpha \big( [\mathfrak{w}_\alpha]\big)\;&=\;\mathtt{}\W_\alpha \big(\mathtt{exp}\circ(i_+)_*([\mathfrak{p}_{b_+}])\big)\;=\;\mathtt{}\W_\alpha \big(\imath_*\circ\mathtt{exp}\big([\mathfrak{p}_{b_+}])\big)\\
    &=\;{\rm Ch}_{\mathcal{T},\tau}([\mathfrak{p}_{b_+}])
\end{split}
\end{equation*}
where $\mathfrak{p}_{b+}$ is the Power-Rieffel projection of  $\A_{b_+}$. Moreover, since there exist $A\in {\rm SO}(2)$\footnote{The group of $2\times 2$  orthogonal matrices on $\R$ with determinant $1.$} such that $Av_\alpha=(1,0)$ and $Av_\alpha^\bot=(0,1)$, then \cite[Proposition 3.4.3]{Tom} implies that 
$$\W_\alpha([\mathfrak{w}_\alpha])\;=\;{\rm Ch}_{\mathcal{T},\tau}([\mathfrak{p}_{b_+}])\;=\;\mathcal{T}\big(\mathfrak{p}_{b_+}[\nabla_1\mathfrak{p}_{b_+},\nabla_2\mathfrak{p}_{b_+}]\big)\;=\;1$$
In the last equality, we used the fact that the Power-Rieffel projection has a Chern number equal to $1$ \cite{Con}. The rational case follows with the same argument by using in this case the following commutative diagram
\begin{equation}
        \xymatrix{
 0\ar[r]&\tilde{\mathfrak{I}}_{\alpha}^{b_+}\ar[r]^{i}\ar[d]_{\imath'} & \tilde{\A}_{\alpha,+}\ar[r]^{{\rm ev}_+}\ar[d]_{\hat{\imath}'}& \A_{b_+}\ar[r]\ar[d]_{i_+}&0 \\
  0\ar[r]&\I_\alpha\ar[r]^{i} & \A_{{\alpha,b'}}\ar[r]^{{\rm ev}}& \A_{\rm bulk}\ar[r]&0 
 }
\end{equation}

\end{proof}

\end{document}